\documentclass[11pt]{article}
\usepackage{setspace}
\usepackage[numbers]{natbib}
\usepackage{amsmath,amssymb,amsthm,mathtools}
\usepackage[margin=0.75in]{geometry}
\usepackage{pdflscape}
\usepackage{booktabs}
\usepackage{tabularx}
\usepackage{array}
\usepackage{float}
\usepackage{graphicx}
\usepackage{placeins}
\usepackage{appendix}
\usepackage{enumitem}
\usepackage{algorithm}
\usepackage{algorithmic}
\usepackage{tikz}
\usetikzlibrary{arrows.meta}
\usetikzlibrary{decorations.pathreplacing}
\usepackage{tikz}
\usepackage{caption}
\usepackage{comment}
\usepackage{bbm}
\newcommand{\mb}[1]{\mbox{\boldmath $#1$}}

\newtheorem{theorem}{Theorem}
\newtheorem{lemma}{Lemma}
\newtheorem{remark}{Remark}

\newtheorem{corollary}{Corollary}

\newtheorem{definition}{Definition}

\DeclareMathOperator*{\argmax}{arg\,max}
\usepackage{hyperref}

\newenvironment{keyword}
{\par\noindent\textbf{Keywords:} }{\par}

\newcounter{appproof}[section]
\renewcommand{\theappproof}{\thesection.\arabic{appproof}}

\newenvironment{appproof}[1][]{%
    \refstepcounter{appproof}%
    \begin{proof}[Proof~\theappproof\if\relax\detokenize{#1}\relax\else~(#1)\fi]%
}{%
    \end{proof}%
}

\title{
Dimension Dependent Correlation Gap Bounds 
under Restricted Independence
}

\author{
Arjun Ramachandra\thanks{Decision Sciences Area, Indian Institute of Management Bangalore, India. Email: \texttt{arjun.ramachandra@iimb.ac.in}}
}

\date{June 2026}

\begin{document}
\maketitle
\begin{abstract}
The pairwise independent correlation gap is defined as the ratio of the maximum expected value of a set function under arbitrary dependence to the maximum expected value when the underlying random elements are restricted to be pairwise independent, and hence measures the loss in expected value under this independence restriction. Under the more restrictive assumption of mutual independence, \cite{Agrawal2012} proved that this gap is universally bounded by $e/(e-1)$ for the class of monotone submodular set functions. With pairwise independence, a tighter upper bound of $4/3$ was established in \cite{ramachandra_special_cases_OR_letters} for several special cases, including $n=3$, and conjectured to hold universally across all dimensions and marginal probabilities for the class of monotone submodular functions. However, a recent AI-assisted counterexample in \cite{ramachandra_natarajan_counter_2026} disproved this conjecture by showing that the $4/3$ bound fails for $n=5$, leaving the validity of the bound for $n=4$ and the tight worst case bound for pairwise independence open.

We resolve both questions in this paper. First, for $n=4$, we establish that the $4/3$ bound holds universally and is tight using an AI-assisted proof combining theoretical analysis and computational verification. 
The proof combines a structural characterization of optimal numerator vertices, permutation symmetry, cone certificate systems, Bernstein polynomial representations, recursive simplex subdivision, and a computational verification of $2,745$ Bernstein coefficient systems. 
Second, we show that the worst case pairwise independent correlation gap
attains $e/(e-1)$ asymptotically by constructing an instance with identical
marginal probabilities and defining a monotone submodular union coverage function
on a ground set partitioned into $m$ blocks. The key condition is that the number of blocks grows
sublinearly with the ground set size, so that both the number of blocks and
the size of each block tend to infinity.
The result follows by constructing a feasible solution to a scaled asymptotic reduced dual
of the pairwise independent linear program with identical marginals. The asymptotic result immediately extends to $t$-wise independent 
($t\ge2$) random elements, since $t$-wise independence implies pairwise independence. 
The key insight of this paper is that the pairwise independent correlation gap exhibits a
dimension dependent worst case bound, with $n=4$ being the largest dimension
for which the $4/3$ bound holds, while the worst case gap fails to improve upon the classical
$e/(e-1)$ bound, attaining it asymptotically.
Thus, pairwise independence, despite being the least restrictive form of independence in the $t$-wise
independence hierarchy
can be as restrictive as mutual independence in the worst case.
\end{abstract}
\section{Introduction}
Let $N=\{1,2,\ldots,n\}$ denote a ground set with $n$ elements and $f: 2^{N} \rightarrow \mathbb{R}_+$ be a nonnegative monotone submodular set function defined on the power set $2^N$. A function is monotone (non-decreasing) if $ f(S) \leq f(T)$ for all $S \subseteq T$ and submodular if $f(S) + f(T) \geq f(S \cap T) +  f(S \cup T)$ for all $S, T \subseteq N$ or equivalently $f(S \cup \{i\}) - f(S) \geq f(T \cup \{i\}) -  f(T)$ for all $S \subseteq T$, $i \in N \backslash T.$ 
For a given marginal probability vector $\mb{x}=(x_1,\ldots,x_n)$ and a monotone submodular function $f$, the correlation gap compares the maximum expected value under arbitrary dependence with the expected value obtained under independence by computing their ratio. 
The numerator in this ratio is referred to as the concave closure
$f^{+}:[0,1]^n\rightarrow\mathbb{R}_+$ and can be viewed as a continous extension of the set function $f$, defined as the maximum expected
value of the function over all joint distributions $\mb{u}$ in which
each element $i\in N$ is selected with probability $x_i$, \emph{i.e.},
\begin{equation}
\label{eq:concave_closure}
\begin{array}{rl}
f^{+}(\mb{x})
=
\displaystyle
\max_{\mb{u}}
&
\mathbb{E}_{\mb{u}}[f]=\displaystyle
\sum_{S\subseteq N}u_S f(S)
\\[4ex]
\mathrm{s.t.}
&
\displaystyle
\sum_{S\subseteq N}u_S=1,
\\[4ex]
&
\displaystyle
\sum_{S\ni i}u_S=x_i,
\qquad \forall i\in N,
\\[4ex]
&
\displaystyle
u_S\ge0,
\qquad \forall S\subseteq N.
\end{array}
\end{equation}
where $u_S$ denotes the probability mass assigned to the subset
$S\subseteq N$. Here, the symbol $u$ is used to denote the univariate information
specified by the marginal probabilities, with no restrictions imposed on
higher order dependence. The denominator in the correlation gap ratio is referred to as the multilinear extension defined as $F(\mb{x})=\sum_{S \subseteq N}f(S)\left(\prod_{i \in S} x_i \prod_{i \notin S} (1-x_i)\right)$ which is the expected function value assuming the random elements to be mutually independent. The correlation gap $f^{+}(\mb{x})/F(\mb{x})$ can be interpreted as the loss of optimality incurred by ignoring the higher order dependence (correlations) and imposing independence instead and was shown in~\cite{Agrawal2012} to be universally bounded by \(e/(e-1)\) for the class of monotone submodular functions. It has found widespread applications in stochastic optimization including facility location, Steiner tree problems, and social welfare maximization with combinatorial auctions~\citep{Agrawal2012}, mechanism design~\citep{Yan2011}, distributionally robust bottleneck optimization \citep{ramachandra_special_cases_OR_letters}, prophet inequalities~\citep{ChekuriLivanos2024}, matroid theory~\citep{Husic2025} and several submodular optimization problems~\citep{Chekuri2014}. Given the restrictive nature of the mutual independence assumption and motivated by whether relaxing mutual independence to weaker notions of independence can further reduce this gap, \cite{ramachandra_natarajan_SIDMA} introduced the pairwise independent extension of the set function $f$ given by $f^{++}: [0,1]^n \rightarrow \mathbb{R}_+$. In addition to the marginal probability information, this extension assumes that the random elements are pairwise independent and can be expressed as:
\begin{equation}
\label{eq:pairwise_concave_closure}
\begin{array}{rl}
f^{++}(\mb{x})
=
\displaystyle
\max_{\mb{p}}
&
\mathbb{E}_{\mb{p}}[f]=\displaystyle
\sum_{S\subseteq N}p_s f(S)
\\[4ex]
\mathrm{s.t.}
&
\displaystyle
\sum_{S\subseteq N}p_s=1,
\\[4ex]
&
\displaystyle
\sum_{S\ni i}p_s=x_i,
\qquad \forall i\in N,
\\[4ex]
&
\displaystyle
\sum_{S\supseteq\{i,j\}}p_s=x_ix_j,
\qquad \forall i<j,\; i \in N,\; j \in N
\\[4ex]
&
\displaystyle
p_s\ge0,
\qquad \forall S\subseteq N.
\end{array}
\end{equation}
where $p_S$ denotes the probability mass assigned to the subset
$S\subseteq N$ with the symbol $p$ used to denote the pairwise information specified by the marginal and pairwise probabilities.
Pairwise independence is a restrictive notion of independence in the sense that it imposes only pairwise constraints while allowing arbitrary higher order dependencies. The sparse support of these distributions also has proved to be useful in
efficient derandomization of algorithms for NP-hard optimization problems
\cite{lubywigderson}. The pairwise independent correlation gap, defined in the same paper of \cite{ramachandra_natarajan_SIDMA} as the ratio
$
f^{+}(\mb{x})/f^{++}(\mb{x}),
$
was expected to admit a sharper upper bound than the $e/(e-1)$ bound on $
f^{+}(\mb{x})/F(\mb{x})
$. This is because the pairwise independent extension in \eqref{eq:pairwise_concave_closure} includes the mutually independent distribution as a feasible solution, which implies that 
$$
 f^{+}(\mb{x})/f^{++}(\mb{x}) \le f^{+}(\mb{x})/F(\mb{x}) \le e/(e-1).
$$
Unfortunately, computing
both $f^{+}(\mb{x})$ and $f^{++}(\mb{x})$ is NP-hard in the worst case for monotone submodular functions (see \cite{Agrawal2012}). In particular, unlike $F(\mb{x})$, which is evaluated for
the mutually independent distribution irrespective of the function $f$, the optimal probability distributions that attain the $f^{+}(\mb{x})$ and $f^{++}(\mb{x})$ are not oblivious
to the function $f$  wich makes the analysis challenging even for small dimensions such as $n=3$ (see \cite{ramachandra_special_cases_OR_letters}).

Despite this challenge, a reduced upper bound of $4/3$ on 
$f^{+}(\mb{x})/f^{++}(\mb{x})$ has been proven  in special cases. For the specific submodular function
$f(S)=\min(|S|,1)$, with arbitrary dimension $n$ and any marginal probabilities $\mb{x} \in [0,1]^n$, \cite{ramachandra_natarajan_SIDMA} proved that the $4/3$ upper bound was tight. The analysis was further extended in \cite{ramachandra_special_cases_OR_letters}, showing that the $4/3$ bound holds for all monotone submodular
functions in two cases: (a) $n=3$ with unrestricted marginal
probabilities, and (b) general $n$ with restricted marginal
probabilities. However, unlike the $e/(e-1)$ bound, the universal validity of the $4/3$ bound for arbitrary monotone submodular functions, dimensions $n$, and marginal probabilities $\mb{x}$ remained an open question, leading \cite{ramachandra_special_cases_OR_letters} to conjecture that the $4/3$ bound holds universally.
\section{Recent Developments and Our Contributions}
More recently, the universal $4/3$ conjecture was disproved in \cite{ramachandra_natarajan_counter_2026}, by constructing an AI-assisted counterexample for $n=5$ Bernoulli random elements that attains a pairwise independent correlation gap of
\[
\frac{640}{479} \approx 1.336 > \frac{4}{3}
\]
with a  monotone submodular coverage function. The $n=5$ counterexample can be extended to every dimension $n> 5$
by padding it with Bernoulli variables having marginal probability zero.
Such variables are identically zero and hence neither appear in the support
of the optimal joint distribution nor affect the marginal
constraints or the optimal objective. This means that the worst case ratio lies in the interval $
\left[{640/479},\,e/(e-1)\right]$ for $n \ge 5$. Consequently, as claimed in \cite{ramachandra_natarajan_counter_2026}, $n=4$ is the only dimension for which the validity of the $4/3$ bound remained unresolved. On the other hand, the counterexample for $n\geq5$ raises a broader question: if the $4/3$ bound fails for
$n\geq5$, what then is a universal upper bound on the pairwise
independent correlation gap across dimensions $n$ ? Can it beat the $e/(e-1)$ bound in the worst case? We resolve both questions in this paper, the first in the affirmative and the second in the negative, by establishing that:
\begin{enumerate}
    \item  The $4/3$ upper bound survives for $n=4$
    using an AI-assisted proof combining theoretical analysis and computational verification, and is tight.
    
    \item  The worst case
    pairwise independent correlation gap  attains
    $e/(e-1)$ asymptotically as $n \to \infty$, the same worst case bound established in \cite{Agrawal2012} under mutual independence. As an immediate consequence, since every $t$-wise independent distribution ($t\ge 2$) is necessarily pairwise independent, our result implies that the worst case $t$-wise independent correlation gap is also exactly $e/(e-1)$.
\end{enumerate} 
These two results are fundamentally different in nature: The first result for $n=4$, involves
providing a universal certification showing that the $4/3$ bound holds for
every monotone submodular function and every marginal probability vector
$\mb{x}\in[0,1]^4$. The second result, on the other hand, provides an asymptotic
worst case construction showing that $e/(e-1)$ is asymptotically attained as $n \to \infty$, implying that pairwise independence, despite being the least restrictive form of independence in the $t$-wise
independence hierarchy, can be as restrictive as mutual independence in the \emph{worst case}. Together, these results add to the literature on the characterization of the worst case pairwise independent correlation gap across dimensions for the class of monotone submodular functions. 

\begin{table}[htbp]
\centering
\caption{Tight worst case pairwise independent correlation gap for monotone submodular functions}
\label{tab:pairwise_gap_dimension}

\begin{tabular}{|c|c|c|}
\hline
Dimension & Tight worst case bound & Reference \\
\hline
\rule{0pt}{4.5ex}
$n=2,3$
& $\displaystyle \frac{4}{3}$
& \cite{ramachandra_special_cases_OR_letters}
\rule[-2.5ex]{0pt}{0pt} \\
\hline
\rule{0pt}{4.5ex}
$n=4$
& $\displaystyle \frac{4}{3}$
& This paper
\rule[-2.5ex]{0pt}{0pt} \\
\hline
\rule{0pt}{4.5ex}
$n\geq5$
& $\displaystyle
\left[\frac{640}{479},\,\frac{e}{e-1}\right]$
& \cite{ramachandra_natarajan_counter_2026}
\rule[-2.5ex]{0pt}{0pt} \\
\hline
\rule{0pt}{4.5ex}
$n\to\infty$
& $\displaystyle \frac{e}{e-1}$
& This paper
\rule[-2.5ex]{0pt}{0pt} \\
\hline
\end{tabular}
\end{table}

Table~\ref{tab:pairwise_gap_dimension} summarizes the progression of
results on the worst case pairwise independent correlation gap and
presents the current state of knowledge across dimensions $n$. The rest of the paper is organized as follows: Section \ref{sec:n=4} establishes the tightness of the $4/3$ bound for $n=4$ random elements while Section \ref{sec:asymptotic_pw_correlation_gap} establishes the asymptotic attainment of the $e/(e-1)$ bound as $n \to \infty$ and Section \ref{sec:conclusion} summarizes the key contributions of this work with possible future directions. Proofs are relegated to the appendix whenever possible and are retained in
the main body only when they are essential for understanding subsequent
results.
\section{A $4/3$ bound for $n=4$ random elements}\label{sec:n=4}
For $n=4$ random elements, denote the ground set as $N_4=\{1,2,3,4\}$. 
\begin{theorem}[Worst case pairwise independent correlation gap for $n=4$]\label{thm:n=4_main_result}
 For every nonnegative monotone
submodular function $f:2^{N_4}\rightarrow\mathbb{R}_+$ and every
$\mb{x}\in[0,1]^4$,
\[
\frac{f^+(\mb{x})}{f^{++}(\mb{x})}\leq \frac{4}{3}
\]
and the bound is tight.
\end{theorem}
\begin{proof}

A brief overview of the proof is outlined in the seven steps below:

\begin{enumerate}[label=(\roman*)]

\item We first recall a standard result from LP theory in Lemma \ref{lem:vertex_support_characterization}, which states that an optimal numerator vertex in $4$ dimensions is supported on at most five subsets. Further, if an optimal vertex is degenerate with smaller support, Corollary \ref{cor:degenerate_extension} shows that it can be extended to a nonsingular five-subset basis.

\item For a chosen optimal numerator vertex corresponding to one of the $3008$ nonsingular five-subset bases, Corollary \ref{cor:common_convex_coefficients} shows that the marginal vector $\mb{x}$ lies in the incidence simplex corresponding to the five incidence vectors of this basis. The set of all incidence simplices corresponding to these bases form an overlapping cover of the entire marginal probability space $[0,1]^4$. 

\item  Lemma~\ref{lem:cone_certificate} provides a sufficient cone certificate condition for the $4/3$ bound to hold for every function in the cone of nonnegative monotone submodular functions and every $\mb{x}\in[0,1]^4$. The key idea is to consider primal, rather than dual, certificates for the numerator, allowing both the numerator and denominator certificates to be formulated in the same primal space and thereby incorporated into a common cone certificate system. Since the optimal numerator vertex can correspond to any one of the 3008 nonsingular five-subset bases, a certificate of feasibility must be provided for each one of them with marginal probabilities restricted to lie in the corresponding incidence simplex.

\item To overcome the challenge of generating certificates over the continuum of marginal vectors for each such basis, we use a degree-two Bernstein basis to express the probability distributions and other relevant parameters as quadratic functions of the convex coefficients, yielding $15$ coefficient blocks corresponding to the
pairs $0\le a\le b\le4$. For each such coefficient block and a chosen nonsingular five-subset basis, Corollary \ref{cor:sufficient_certificates} shows that it is sufficient to check for feasibility of a Bernstein coefficient system  defined in Lemma \ref{lem:bernstein_coefficient_system} to ensure a feasible cone certificate.

\item Further, we exploit permutation symmetry to show that the $3,008$ nonsingular five-subset bases can be reduced to $183$ orbits, where bases in the same orbit are equivalent up to a permutation of the four elements. Lemma \ref{lem:symmetry_permutation_invariance} shows that it suffices to solve each of the the $15$ Bernstein coefficient systems for one representative from each orbit.

\item  To handle representative orbit simplices on which the Bernstein coefficient system is infeasible, we provide an Algorithm~\ref{alg:bernstein_subdivision} that recursively subdivides each such simplex along its longest edge until all terminal leaf simplices admit feasible certificates, thereby providing a certified cover of the entire simplex.

\item 
Finally, we implemented Algorithm~\ref{alg:bernstein_subdivision} in code and used linear programming to computationally verify the feasibility of all $15\times183=2,745$ Bernstein coefficient systems corresponding to the $15$ coefficient blocks and $183$ permutation orbit representatives. All systems were found to be feasible, yielding a total of $4,476$ successful leaf simplices and thereby establishing the $4/3$ bound for $n=4$ as claimed.
\end{enumerate}
\end{proof} 
\paragraph{Declaration of AI-assisted research methodology for the $n=4$ proof:} The $n=4$ proof strategy was initially explored through a human-AI workflow involving GPT 5.6 and Fable 5, including the cone certificate formulation and the use of a Bernstein representation to avoid checking certificates separately at every marginal probability vector. The author subsequently verified, corrected, and refined the resulting strategy, in particular by requiring the numerator distribution in the cone certificate to correspond to an optimal primal vertex, and developed the resulting systematic treatment of the possible optimal vertices. The subsequent proof development incorporated permutation symmetry and recursive simplex subdivision along with independent verification of computational results. 

\begin{remark}[Negative dependence is fragile in high dimensions]
The structure of submodularity tends to encourage negative dependence, and
our $n=4$ result resonates with the broader observation in the literature that negative
dependence is delicate and need not persist beyond low dimensional regimes. In matroid theory, for example,~\cite{wagner_rank_2004} showed that every rank-$3$ matroid has the
Rayleigh property, which corresponds to pairwise negative correlation,
whereas rank-$4$ matroids need not have this property.~\cite{huh_schroter_wang_2022} provided an example to illustrate a positively correlated
pair in a rank-$4$ transversal matroid.
\end{remark}

 \begin{remark}[Advantages of a primal based proof technique]
Unlike \cite{ramachandra_special_cases_OR_letters}, where the $4/3$ bound for $n=3$ was proved analytically using the dual of the concave
closure for the numerator (see Theorem ~5 therein), the primal based approach in our proof of Theorem \ref{thm:n=4_main_result} works
directly with $183$ representative permutation orbits of the $3,008$ possible optimal primal numerator vertices. 
The denominator, however, is handled directly through a primal feasible
pairwise independent distribution, similar to \cite{ramachandra_special_cases_OR_letters}. The key advantage of a primal based approach for the numerator is that it
allows the numerator and denominator certificates to be formulated in the
same primal space, so that a common cone certificate system can be used, as in  Lemma \ref{lem:cone_certificate}.
In contrast, the approach in \cite{ramachandra_special_cases_OR_letters}
requires identifying dual feasible solutions for the numerator and primal
feasible solutions for the denominator across every interaction between the
partition blocks of the marginal probability space and those of the
submodular function cone.
\end{remark}

\begin{remark}[Tightness of $4/3$ bound for $n=4$]
The $4/3$ bound in Theorem~\ref{thm:n=4_main_result} for $n=4$ is attained
for the submodular function $f(S)=\min\{|S|,1\}$, $S\subseteq N_4$, with
marginal probabilities satisfying
$
x_1+x_2+x_3=\frac12,\; x_4=\frac12.
$
This follows directly from Proposition~2.5 in
\cite{ramachandra_natarajan_SIDMA}, which establishes the same tightness
result for general $n$.
\end{remark}
\noindent The remainder of this section details the proof of Theorem~\ref{thm:n=4_main_result} in parts, following the steps outlined above.
\subsection{Preliminary results leading to proof of Theorem~\ref{thm:n=4_main_result}}
For any $S\subseteq N_4$, let
\[
\mb{v}_S
=
\left(
\mathbbm{1}_{\{1\in S\}},
\mathbbm{1}_{\{2\in S\}},
\mathbbm{1}_{\{3\in S\}},
\mathbbm{1}_{\{4\in S\}}
\right)^\top
\in\{0,1\}^4,
\]
denote the incidence vector of $S$, where $\mathbbm{1}_{\{\cdot\}}$ denotes the indicator function. Define the matrix
\[
A=
\begin{bmatrix}
1 & 1 & \cdots & 1\\
\mb{v}_{\emptyset} & \mb{v}_{\{1\}} & \cdots & \mb{v}_{N_4}
\end{bmatrix}
\in\mathbb{R}^{5\times16},
\]
where the columns are indexed by the $16$ subsets of $N_4$. Then for any  fixed marginal vector $\mb{x}\in[0,1]^4$, the
 feasible region of the concave closure \eqref{eq:concave_closure} can be expressed as:
\[
\mathcal{F}_{\mathrm{U}}(\mb{x})
:=
\left\{
\mb{u}\in\mathbb{R}^{16}_+:
A\mb{u}=
\begin{pmatrix}
1\\
\mb{x}
\end{pmatrix}
\right\},
\]
which is non-empty, since the mutually independent distribution $\mb{u}_{\textrm{ind}}(S)=\prod_{i \in S}x_i \;\prod_{i \notin S}(1-x_i)$ satisfies $\mb{u}_{\textrm{ind}} \in \mathcal{F}_{\mathrm{U}}(\mb{x})$. The optimal objective of the concace closure is then:
\[
f^+(\mb{x})
=
\max_{\mb{u}\in\mathcal{F}_{\mathrm{U}}(\mb{x})}
\mb{f}^{\top}\mb{u}
\]
where $\mb{f}\in\mathbb{R}^{16}_+$ denotes the vector representation of the non-negative monotone submodular
set function $f:2^{N_4}\to\mathbb{R}$.
Denote by $2^{N_4}$ the power set of $N_4$ containing $16$ possible subsets. Let $J=\operatorname{supp}(\mb{u})\subseteq 2^{N_4}$ denote the support of
the feasible distribution $\mb{u}$, \emph{i.e.}, the collection of subsets
$S\subseteq N_4$ to which $\mb{u}$ assigns positive probability. The next lemma recalls a standard result from linear programming theory characterizing the support of a vertex distribution $\mb{u}\in\mathcal{F}_{\mathrm U}(\mb{x})$.
\begin{lemma}[Vertex characterization by support]
\label{lem:vertex_support_characterization}
A feasible distribution $\mb{u}\in\mathcal{F}_{\mathrm U}(\mb{x})$ is a vertex if and
only if the columns of $A$ corresponding to
$\operatorname{supp}(\mb{u})$ are linearly independent. In particular,
every vertex of $\mathcal{F}_{\mathrm U}(\mb{x})$ is supported on at most
five subsets of $N_4$.
\end{lemma}

\begin{proof}
The proof is relegated to Appendix~\ref{proof:lem_vertex_support_characterization}. 
\end{proof}
We call a collection of five subsets
$
B=\{S_0,S_1,S_2,S_3,S_4\},
\; S_i\subseteq N_4,
$
a \emph{nonsingular five subset basis} if the corresponding incidence
vectors
$
\mb{v}_0,\mb{v}_1,\mb{v}_2,\mb{v}_3,\mb{v}_4\in\{0,1\}^4
$
are affinely independent, or equivalently, the associated submatrix
\[
A_B=
\begin{bmatrix}
1 & 1 & 1 & 1 & 1\\
\mb{v}_0 & \mb{v}_1 & \mb{v}_2 & \mb{v}_3 & \mb{v}_4
\end{bmatrix}
\in\mathbb{R}^{5\times5}
\]
is nonsingular. We next establish that the support of a degenerate vertex can be extended to a nonsingular
five subset basis. 
\begin{corollary}[Extension of a degenerate vertex to a nonsingular five subset basis]
\label{cor:degenerate_extension}
Let $\mb{u}$ be a degenerate vertex of $\mathcal{F}_{\mathrm{U}}(\mb{x})$ with
$|\operatorname{supp}(\mb{u})|<5$. Then the corresponding columns of $A$
can be extended to a set of five linearly independent columns of $A$,
corresponding to a nonsingular five subset basis.
\end{corollary}

\begin{proof}
By Lemma~\ref{lem:vertex_support_characterization}, for any vertex $\mb{u}$, the columns corresponding to 
$\operatorname{supp}(\mb{u})$ are linearly independent. Since
$\operatorname{rank}(A)=5$, they can be extended to five linearly
independent columns of $A$, yielding a nonsingular five subset basis.
\end{proof}
Let $V=\{\mb{v}_0,\ldots,\mb{v}_4\}$
be a set of affinely
independent incidence vectors $\mb{v}_0,\ldots,\mb{v}_4\in\{0,1\}^4$. Define the corresponding incidence simplex as
$
\Delta_{V}
=
\operatorname{conv}\{\mb{v}_0,\ldots,\mb{v}_4\}.
$ 
The next result shows that the optimal numerator vertex
$\mb{u}^\star(\mb{x},\mb{f})$ corresponding to a given marginal probability vector $\mb{x}$ and monotone submodular function vector $\mb{f}$ induces a nonsingular five subset basis $B$
whose incidence simplex $\Delta_V(B)$ contains the marginal probability vector
$\mb{x}$. 
\begin{corollary}[Marginal probability vector in the convex hull of an optimal basis]
\label{cor:common_convex_coefficients}
Given $\mb{x}\in[0,1]^4$ and a monotone submodular function vector $\mb{f}$, let
$\mb{u}^\star(\mb{x},\mb{f})$ denote an optimal vertex of
$\mathcal{F}^{\mathrm{U}}(\mb{x})$. Then there exists a nonsingular
five subset basis
$
B=\{S_0,S_1,S_2,S_3,S_4\},\; S_i\subseteq N_4,
$
with corresponding affinely independent incidence vectors
$\mb{v}_0,\ldots,\mb{v}_4\in\{0,1\}^4$ and convex coefficients
$\eta_i\ge0$, $\sum_{i=0}^4\eta_i=1$, such that
\[
\mb{u}^\star(\mb{x},\mb{f})
=
\sum_{i=0}^4\eta_i\mb{e}_{S_i},
\qquad
\mb{x}
=
\sum_{i=0}^4\eta_i\mb{v}_i.
\]
where $\mb{e}_S\in\mathbb{R}^{16}$ denote the unit vector corresponding to
$S\subseteq N_4$. In particular, $\mb{x}\in\Delta_V
(B)=\operatorname{conv}\{\mb{v}_0,\ldots,\mb{v}_4\}$.
\end{corollary}

\begin{proof}
Since $\mathcal{F}_{\mathrm U}(\mb{x})\neq \emptyset$, an optimal vertex must exist for any $\mb{x} \in [0,1]^4$. By Lemma~\ref{lem:vertex_support_characterization}, the optimal vertex
$\mb{u}^\star(\mb{x},\mb{f})$ is supported on at most five subsets. By
Corollary~\ref{cor:degenerate_extension}, this support can be extended
to a nonsingular five subset basis
$B=\{S_0,\ldots,S_4\}$. Hence
$\mb{u}^\star(\mb{x},\mb{f})=\sum_{i=0}^4\eta_i\mb{e}_{S_i}$ for convex
coefficients $\eta_i$, with $\eta_i=0$ for any added subset in the extension. Since
$A\mb{u}^\star(\mb{x},\mb{f})=(1,\mb{x})^\top$ and
$A\mb{e}_{S_i}=(1,\mb{v}_i)^\top$, we obtain
\[
\begin{pmatrix}1\\ \mb{x}\end{pmatrix}
=
\sum_{i=0}^4\eta_i
\begin{pmatrix}1\\ \mb{v}_i\end{pmatrix},
\]
and therefore $\mb{x}=\sum_{i=0}^4\eta_i\mb{v}_i$, implying that
$\mb{x}\in\Delta_V(B)$.
\end{proof}

\begin{corollary}[Coverage and overlap of nonsingular five-subset simplices]
\label{cor:nonsingular_simplices_cover}
Let $\mathcal{B}_{\mathrm{NS}}$ denote the collection of all nonsingular
five-subset bases, and let $\Delta_V(B)$ denote the incidence simplex
corresponding to each $B\in\mathcal{B}_{\mathrm{NS}}$. Then the set of all such incidences simplices form an overlapping cover of the entire marginal probability space, \emph{i.e.},
\[
[0,1]^4
=
\bigcup_{B\in\mathcal{B}_{\mathrm{NS}}}\Delta_V(B),
\]
and distinct nonsingular incidence simplices can have
non-empty intersections.
\end{corollary}

\begin{proof}
The proof is relegated to Appendix~\ref{proof:cor_nonsingular_simplices_cover}. 
\end{proof}

We next turn to the denominator of the correlation gap and characterize pairwise independent  distributions by appending the six pairwise moment constraints to the
total probability and marginal constraints. For $S\subseteq N_4$, let
\[
\mb{w}_S
=
\left(
\mathbbm{1}_{\{1,2\}\in S},
\mathbbm{1}_{\{1,3\}\in S},
\mathbbm{1}_{\{1,4\}\in S},
\mathbbm{1}_{\{2,3\}\in S},
\mathbbm{1}_{\{2,4\}\in S},
\mathbbm{1}_{\{3,4\}\in S}
\right)^\top
\in\{0,1\}^6.
\]
denote the vector of pairwise incidences of $S$. For a marginal vector $
\mb{x}=(x_1,x_2,x_3,x_4)^\top\in[0,1]^4$, define the pairwise product vector
\[
\mb{y}(\mb{x})
=
\left(
x_1x_2,x_1x_3,x_1x_4,
x_2x_3,x_2x_4,x_3x_4
\right)^\top
\in\mathbb{R}^6.
\]
and the appended matrix
\[
A^{\mathrm P}=
\begin{bmatrix}
1 & 1 & \cdots & 1\\
\mb{v}_{\emptyset} & \mb{v}_{\{1\}} & \cdots & \mb{v}_{N_4}\\
\mb{w}_{\emptyset} & \mb{w}_{\{1\}} & \cdots & \mb{w}_{N_4}
\end{bmatrix}
\in\mathbb{R}^{11\times16}.
\]
Then, for any fixed marginal vector $\mb{x}\in[0,1]^4$, the
 feasible region of the pairwise independent extension  \eqref{eq:pairwise_concave_closure} can be expressed as:
\begin{equation}\label{eq:pairwise_feasible_region}
\begin{array}{cc}
\mathcal{F}_{\mathrm P}(\mb{x})
:=
\left\{
\mb{p}\in\mathbb{R}^{16}_+:
A^{\mathrm P}\mb{p}
=
\begin{pmatrix}
1\\
\mb{x}\\
\mb{y}(\mb{x})
\end{pmatrix}
\right\}.
\end{array}
\end{equation}
which is non-empty, since the mutually independent distribution is also pairwise independent. 
The corresponding optimal objective is
\[
f_{\mathrm P}^{+}(\mb{x})
=
\max_{\mb{p}\in\mathcal{F}_{\mathrm P}(\mb{x})}
\mb{f}^{\top}\mb{p}.
\]

\noindent We next characterize the cone of normalized monotone submodular functions
on $N_4$.

\subsection{Cone certificate equation}

Let $\mb{f}\in\mathbb{R}^{16}$ denote the vector representation of a non-negative monotone submodular 
set function $f:2^{N_4}\to\mathbb{R}_+$, with one component corresponding
to each subset of $N_4$. Monotonicity requires
\[
f(S\cup\{i\})-f(S)\ge0,
\qquad
\forall i\in N_4,\quad \forall S\subseteq N_4\setminus\{i\},
\]
which can be represented by
$\sum_{i\in N_4}|2^{N_4\setminus\{i\}}|=4\cdot2^3=32$
inequalities (where $2^S$ denotes the power set of $S$), since each element
$i\in N_4$ can be added to any subset of the remaining three elements.
Submodularity requires
\[
f(S\cup\{i\})+f(S\cup\{j\})
-f(S)-f(S\cup\{i,j\})\ge0,\quad \forall i,j\in N_4,\quad
\forall S\subseteq N_4\setminus\{i,j\},\quad i<j,
\]
which can similarly be represented by
$\sum_{i,j\in N_4,\;i<j}|2^{N_4\setminus\{i,j\}}|
=6\cdot2^2=24$
inequalities. Thus, any monotone submodular function can be represented by
$32+24=56$ inequalities, compactly written as
\[
G\mb{f}\ge0,
\qquad
G\in\{-1,0,1\}^{56\times16},
\]
where each row of $G$ is the coefficient vector of one of the
monotonicity or submodularity inequalities above. The cone of monotone
submodular functions normalized at the empty set is therefore
\[
\mathcal{C}
=
\left\{
\mb{f}\in\mathbb{R}^{16}:
G\mb{f}\ge0,\;
f(\emptyset)=0
\right\}.
\]
We next establish a cone certificate equation that provides a sufficient
condition for the desired $4/3$ bound on the correlation gap.

\begin{lemma}[Cone certificate equation]
\label{lem:cone_certificate}
Fix $\mb{x}\in[0,1]^4$ and a normalized monotone submodular function
$\mb{f}\in\mathcal{C}$. Let $\mb{u}^\star(\mb{x},\mb{f})$ denote an
optimal vertex of $\mathcal{F}_{\mathrm U}(\mb{x})$. Suppose there exists
a certificate $(\mb{p},\mb{\gamma},c)$ with
\[
\mb{p}\in\mathcal{F}_{\mathrm P}(\mb{x}),
\qquad
\mb{\gamma}\in\mathbb{R}^{56}_+,
\qquad
c\in\mathbb{R},
\]
satisfying
\begin{equation}
\frac{4}{3}\mb{p}-\mb{u}^\star
=
G^\top\mb{\gamma}
+
c\,\mb{e}_{\emptyset}.
\label{eq:cone_certificate}
\end{equation}
where $\mb{e}_{\emptyset}\in\mathbb{R}^{16}$ denotes the unit vector
corresponding to the empty set. Then
\[
f^+(\mb{x})
\le
\frac{4}{3}f^{++}(\mb{x}).
\]
\end{lemma}

\begin{proof}
We refer to \eqref{eq:cone_certificate} as the \emph{cone certificate
equation}, since $G^\top\mb{\gamma}$, with
$\mb{\gamma}\ge \mb{0}$, belongs to the conic hull of the columns of
$G^\top$ (equivalently, the rows of $G$). Taking the inner product of
\eqref{eq:cone_certificate} with $\mb{f}$ gives
\[
\frac{4}{3}\mb{p}^{\top}\mb{f}
-
\mb{u}^{\star^{\top}}\mb{f}
=
\mb{\gamma}^{\top}G\mb{f}
+
c f(\emptyset).
\]
Since $\mb{f}$ is monotone submodular, we have
$G\mb{f}\ge \mb{0}$
and since $\mb{\gamma}\ge \mb{0}$, we have
$
\mb{\gamma}^{\top}G\mb{f}\ge0.
$
Moreover, since $\mb{f}$ is normalized, $f(\emptyset)=0$. Hence
\[
f^+(\mb{x})=\mb{u}^{\star^{\top}}\mb{f}
\le
\frac{4}{3}\mb{p}^{\top}\mb{f}
\le
\frac{4}{3}f^{++}(\mb{x}),
\]
where the second inequality follows from
$\mb{p}\in\mathcal{F}_{\mathrm P}(\mb{x})$.
\end{proof}
\begin{remark}
The scalar $c$ provides an additional degree of freedom in the cone
certificate equation \eqref{eq:cone_certificate}, without which, the vector
$\frac{4}{3}\mb{p}-\mb{u}^{\star}$ would be forced to lie in the cone generated by the columns of $G^\top$. The term $c\,\mb{e}_{\emptyset}$ relaxes the first coordinate of the
$16$-dimensional vector equation, while having no effect on the resulting
inequality because $f(\emptyset)=0$ for every normalized submodular
function.
\end{remark}

\begin{corollary}[Sufficiency of cone certificates for each nonsingular
five-subset basis]
\label{cor:sufficient_certificates}
Let $\mathcal{B}_{\mathrm{NS}}$ denote the collection of all nonsingular
five-subset bases in four dimensions. There are
$\binom{16}{5}=4,368$ five-subset bases in four dimensions, of which it can be verified that
$|\mathcal{B}_{\mathrm{NS}}|=3,008$ are nonsingular. Then:
\begin{enumerate}[label=(\roman*)]
\item Every possible combination of $\mb{x}$ and
$\mb{u}^\star(\mb{x},\mb{f})$ is associated with some
$B\in\mathcal{B}_{\mathrm{NS}}$, with
$\mb{x}\in\Delta_V(B)$ and $\mb{u}^\star(\mb{x},\mb{f})$ supported on
the corresponding five subsets.

\item To prove a $4/3$ bound on the pairwise independent correlation
gap, it suffices to establish a feasible certificate
$(\mb{p},\mb{\gamma},c)$ satisfying \eqref{eq:cone_certificate} for
every possible combination of $\mb{x}$ and optimal numerator vertex
$\mb{u}^\star(\mb{x},\mb{f})$, which  together correspond to
some $B\in\mathcal{B}_{\mathrm{NS}}$.
\end{enumerate}
\end{corollary}

\begin{proof}
For any $\mb{x}\in[0,1]^4$ and $\mb{f}\in\mathcal{C}$, let
$\mb{u}^\star(\mb{x},\mb{f})$ denote the optimal numerator vertex.
By Lemma~\ref{lem:vertex_support_characterization},
 every
vertex of $\mathcal{F}_{\mathrm U}(\mb{x})$, and hence
$\mb{u}^\star(\mb{x},\mb{f})$, corresponds to a nonsingular five-subset
basis. If $\mb{u}^\star(\mb{x},\mb{f})$ is degenerate, Corollary~\ref{cor:degenerate_extension}
allows its support to be extended to a nonsingular five-subset basis
$B\in\mathcal{B}_{\mathrm{NS}}$, with zero probability assigned to any
added subsets. By Corollary~\ref{cor:common_convex_coefficients}, the
same convex coefficients representing
$\mb{u}^\star(\mb{x},\mb{f})$ on this basis represent $\mb{x}$ as a
convex combination of the corresponding incidence vectors and hence
$\mb{x}\in\Delta_V(B)$. Since the incidence simplices
$\Delta_V(B)$, $B\in\mathcal{B}_{\mathrm{NS}}$, cover $[0,1]^4$ by
Corollary~\ref{cor:nonsingular_simplices_cover}, every possible combination of $\mb{x}$ and $\mb{u}^\star(\mb{x},\mb{f})$  is covered by some
$B\in\mathcal{B}_{\mathrm{NS}}$. Therefore, it suffices to establish
the certificate for every such combination corresponding to a nonsingular five subset basis.
\end{proof}
\begin{remark}
We note that in item~(ii) of Corollary \ref{cor:sufficient_certificates}, no separate verification over $\mb{f}\in\mathcal{C}$ is
necessary to prove the $4/3$ bound.  This follows from the inner product argument in the proof of Lemma~\ref{lem:cone_certificate}, which ensures
that any certificate satisfying \eqref{eq:cone_certificate} for the
corresponding $\mb{x}$ and $\mb{u}^\star(\mb{x},\mb{f})$ automatically establishes the
$4/3$ bound for every normalized monotone submodular function
$\mb{f}\in\mathcal{C}$.
\end{remark}
Although $\mathcal{B}_{\mathrm{NS}}$ is finite, each simplex
$\Delta_V(B)$ contains a continuum of marginal probability vectors and
hence infinitely many possible combinations of $\mb{x}$ and
$\mb{u}^\star(\mb{x},\mb{f})$, each of which, by Corollary \ref{cor:sufficient_certificates}, would need a cone certificate. We overcome this challenge by using a
Bernstein basis representation, which allows the cone certificate to be
verified simultaneously for every $\mb{x}$ in the simplex, avoiding
verification for each $\mb{x}$ separately.

\subsection{Bernstein basis functions}

For any $\mb{x}\in[0,1]^4$ and corresponding optimal numerator vertex
$\mb{u}^\star(\mb{x},\mb{f})$, Corollary~\ref{cor:common_convex_coefficients}
ensures the existence of a nonsingular five-subset basis
$B=\{S_0,\ldots,S_4\}$ with the corresponding 
incidence simplex $\Delta_V
(B)=\operatorname{conv}\{\mb{v}_0,\ldots,\mb{v}_4\}$ and common convex
coefficients $\eta_0,\ldots,\eta_4$ such that
\[
\mb{x}(\mb{\eta})
=
\sum_{i=0}^4\eta_i\mb{v}_i \in \Delta_V,
\qquad
\mb{u}^\star(\mb{\eta})
=
\sum_{i=0}^4\eta_i\mb{e}_{S_i},
\qquad
\sum_{i=0}^4\eta_i=1,\quad \eta_i\ge0.
\]
Thus, each marginal $x_i$ is affine in $\mb{\eta}$, each
pairwise moment $x_ix_j$ is quadratic in $\mb{\eta}$, and the
optimal numerator distribution $\mb{u}^\star$ is affine in
$\mb{\eta}$. This polynomial dependence on
$\mb{\eta}$ allows the cone certificate to be represented and
verified using Bernstein basis functions.
A degree-two
Bernstein basis
(see \citep{Bernstein1912,Farouki2012}) provides a convenient way to express this dependence and reduce verification
over an entire simplex $\Delta_{V}$ to finitely many coefficient conditions. Specifically, any polynomial $h(\mb{\eta})$ of degree at most two
can be written as
\[
h(\mb{\eta})
=
\sum_{0\le a\le b\le4}
h^{ab}B_{ab}(\mb{\eta}),
\]
where
\[
B_{aa}(\mb{\eta})=\eta_a^2,
\qquad
B_{ab}(\mb{\eta})=2\eta_a\eta_b,
\quad 0\le a\le b\le4.
\]
are the $15$ Bernstein basis functions with Bernstein coefficients $h^{ab}$. Thus, for each simplex $\Delta_{V}$ and $\mb{x} \in \Delta_{V}$, the marginal and pairwise moment vector 
\[
\mb{r}(\mb{\eta})
=
\begin{pmatrix}
1\\
\mb{x}(\mb{\eta})\\
\mb{y}(\mb{\eta})
\end{pmatrix}
\]
in the moment equations of \eqref{eq:pairwise_feasible_region} and the optimal vertex distribution $\mb{u}^*(\mb{\eta})$ in the cone certificate equation \eqref{eq:cone_certificate} can both be
represented in this basis, with the Bernstein coefficients as derived in the next Lemma.
\begin{lemma}[Bernstein coefficients of the marginal moments, pairwise moments, and optimal vertex]
\label{lem:bernstein_coefficients}\qquad
For a given $\mb{x}\in[0,1]^4$ and a corresponding optimal numerator
vertex $\mb{u}^\star(\mb{x},\mb{f})$, let
$B=\{S_0,\ldots,S_4\}$ denote the corresponding nonsingular five-subset
basis, with incidence vectors $\mb{v}_0,\ldots,\mb{v}_4$. Then the degree-two Bernstein coefficients of
$\mb{x}(\mb{\eta})$, $\mb{y}(\mb{\eta})$, and
$\mb{u}^\star(\mb{\eta})$ satisfying
\[
\mb{x}(\mb{\eta})
=
\sum_{i=0}^4\eta_i\mb{v}_i,
\qquad
\mb{u}^\star(\mb{\eta})
=
\sum_{i=0}^4\eta_i\mb{e}_{S_i},
\qquad
\sum_{i=0}^4\eta_i=1,\quad \eta_i\ge0
\] are given by
\begin{equation}
\label{eq:bernstein_coefficients_xyu}
\begin{array}{llllll}
\mb{x}^{aa} &=& \mb{v}_a,
&
\mb{x}^{ab} = \dfrac{\mb{v}_a+\mb{v}_b}{2},
&& a<b,\\[0.3em]
y_{ij}^{aa} &=& v_{ai}v_{aj},
&
y_{ij}^{ab}
=
\dfrac{v_{ai}v_{bj}+v_{bi}v_{aj}}{2},
&& a<b,\quad 1\le i<j\le4,\\[0.3em]
\mb{u}^{{\star}^{aa}} &=& \mb{e}_{S_a},
&
\mb{u}^{{\star}^{ab}}
=
\dfrac{\mb{e}_{S_a}+\mb{e}_{S_b}}{2},
&& a<b.
\end{array}
\end{equation}
\end{lemma}
\begin{proof}
Since $\sum_{i=0}^4\eta_i=1$, the affine representation of $\mb{x}$ can be
written as
\[
\mb{x}(\mb{\eta})
=
\left(\sum_{i=0}^4\eta_i\mb{v}_i\right)
\left(\sum_{j=0}^4\eta_j\right)
=
\sum_{i=0}^4\eta_i^2\mb{v}_i
+
\sum_{i<j}\eta_i\eta_j(\mb{v}_i+\mb{v}_j).
\]
Using $B_{ii}(\mb{\eta})=\eta_i^2$ and
$B_{ij}(\mb{\eta})=2\eta_i\eta_j$ for $i<j$ gives
\[
\mb{x}(\mb{\eta})
=
\sum_{i=0}^4B_{ii}(\mb{\eta})\mb{v}_i
+
\sum_{i<j}B_{ij}(\mb{\eta})
\frac{\mb{v}_i+\mb{v}_j}{2},
\]
which yields the stated coefficients of $\mb{x}$. Similarly,
\[
y_{ij}=x_i x_j
=
\left(\sum_{k=0}^4\eta_kv_{ki}\right)
\left(\sum_{\ell=0}^4\eta_\ell v_{\ell j}\right),
\]
and collecting the diagonal and off-diagonal terms gives the stated coefficients of $\mb{y}$. Finally, similar to $\mb{x}$, the coefficients
of $\mb{u}^\star$ follow from
\[
\mb{u}^\star(\mb{\eta})
=
\left(\sum_{i=0}^4\eta_i\mb{e}_{S_i}\right)
\left(\sum_{j=0}^4\eta_j\right),
\]
which yields the stated coefficients of $\mb{u}^\star$.
\end{proof}

We next seek Bernstein coefficients
$\mb{p}^{ab}$, $\mb{\gamma}^{ab}$, and $c^{ab}$ such that the
moment equations in \eqref{eq:pairwise_feasible_region} and the cone
certificate equation in \eqref{eq:cone_certificate} hold coefficientwise
for each $0\le a\le b\le4$. This coefficientwise feasibility is a
sufficient condition for the corresponding polynomial identities to hold
for every $\mb{\eta}$ in the simplex. Specifically, we represent
the certificate variables as

\begin{equation}\label{eq:bernstein_coefficients_p_lambda_c}
\mb{p}(\mb{\eta})
=
\sum_{0\le a\le b\le4}
\mb{p}^{ab}B_{ab}(\mb{\eta}),
\qquad
\mb{\gamma}(\mb{\eta})
=
\sum_{0\le a\le b\le4}
\mb{\gamma}^{ab}B_{ab}(\mb{\eta}),\qquad
c(\mb{\eta})
=
\sum_{0\le a\le b\le4}
c^{ab}B_{ab}(\mb{\eta}),
\end{equation}
where
\[
\mb{p}^{ab}\in \mathbb{R}^{16}_+,
\qquad
\mb{\gamma}^{ab}\in \mathbb{R}^{56}_+,
\qquad
c^{ab}\in\mathbb{R}.
\]

\begin{lemma}[Bernstein coefficient system]
\label{lem:bernstein_coefficient_system}
Fix $\mb{x}\in[0,1]^4$ and a corresponding optimal numerator
vertex $\mb{u}^\star(\mb{x},\mb{f})$, represented on its associated
nonsingular five-subset basis by
\[
\mb{x}(\mb{\eta})
=
\sum_{i=0}^4\eta_i\mb{v}_i,
\qquad
\mb{u}^\star(\mb{\eta})
=
\sum_{i=0}^4\eta_i\mb{e}_{S_i},
\qquad
\sum_{i=0}^4\eta_i=1,\quad \eta_i\ge0.
\]
For each $0\le a\le b\le4$, let
$\mb{p}^{ab}$,
$\mb{\gamma}^{ab}$, and $c^{ab}$ denote the Bernstein
coefficients of $\mb{p}(\mb{\eta})$, $\mb{\gamma}(\mb{\eta})$, and $c$, respectively from \eqref{eq:bernstein_coefficients_p_lambda_c}.
A sufficient condition for the moment equations in
\eqref{eq:pairwise_feasible_region} and the cone certificate equation
\eqref{eq:cone_certificate} to hold for every
$\mb{\eta}$ is that, for every $0\le a\le b\le4$,
\begin{equation}\label{eq:coeff_wise_27_equations_73_variables}
\begin{aligned}
A_{\mathrm P}\mb{p}^{ab}
&= \mb{r}^{ab},\\
\frac{4}{3}\mb{p}^{ab}
-
G^\top\mb{\gamma}^{ab}
-
c^{ab}\mb{e}_{\emptyset}
&= \mb{u}^{{\star}^{ab}}\\
\mb{p}^{ab},\mb{\gamma}^{ab} 
&\ge \mb{0}.
\end{aligned}
\end{equation}
where $\mb{r}^{ab}$ and $\mb{u}^{{\star}^{ab}}$ are the degree-two
Bernstein coefficients of
$\mb{r}(\mb{\eta})$
and $
\mb{u}^\star(\mb{\eta})$ respectively from
\eqref{eq:bernstein_coefficients_xyu}.
Thus, for each pair $(a,b)$, the coefficient system consists of
$11+16=27$ linear equations in
$16+56+1=73$ unknowns
$(\mb{p}^{ab},\mb{\gamma}^{ab},c^{ab})$.
\end{lemma}

\begin{proof}
Substituting the Bernstein representations of
$\mb{p}(\mb{\eta})$,
$\mb{\gamma}(\mb{\eta})$,
$c(\mb{\eta})$,
$\mb{r}(\mb{\eta})$, and
$\mb{u}^\star(\mb{\eta})$, and using the coefficient equations in \eqref{eq:coeff_wise_27_equations_73_variables}, we obtain
\[
\begin{aligned}
A_{\mathrm P}\mb{p}(\mb{\eta})
&=
\sum_{0\le a\le b\le4}
A_{\mathrm P}\mb{p}^{ab}B_{ab}(\mb{\eta})
=
\sum_{0\le a\le b\le4}
\mb{r}^{ab}B_{ab}(\mb{\eta})
=
\mb{r}(\mb{\eta}),\\
\frac{4}{3}\mb{p}(\mb{\eta})
-
G^\top\mb{\gamma}(\mb{\eta})
-
c(\mb{\eta})\mb{e}_{\emptyset}
&=
\sum_{0\le a\le b\le4}
\left(
\frac{4}{3}\mb{p}^{ab}
-
G^\top\mb{\gamma}^{ab}
-
c^{ab}\mb{e}_{\emptyset}
\right)
B_{ab}(\mb{\eta})\\
&=
\sum_{0\le a\le b\le4}
\mb{u}^{{\star}^{ab}}B_{ab}(\mb{\eta})
=
\mb{u}^\star(\mb{\eta}).
\end{aligned}
\]
 Hence, the first two equation systems in \eqref{eq:coeff_wise_27_equations_73_variables} along with the non-negativity conditions $\mb{p}^{ab},\mb{\gamma}^{ab} \ge \mb{0}$ are sufficient to ensure that $\mb{p}(\mb{\eta})\in \mathcal{F}^{\mathrm{P}}(\mb{x})$ and $
\mb{\gamma}(\mb{\eta})\in \mathbb{R}^{56}_+$. This ensures that the moment equations in \eqref{eq:pairwise_feasible_region} and the cone certificate equation
\eqref{eq:cone_certificate} hold for every
$\mb{\eta}$ and therefore for every marginal probability
vector $\mb{x}$ in the corresponding simplex $\Delta_V$ whose optimal numerator
vertex is $\mb{u}^\star$.
\end{proof}
The Bernstein coefficient system need not be solved separately for all
$3,008$ nonsingular five-subset bases, since many bases are equivalent
up to a permutation of the elements of $N_4$. We next introduce the
concept of an \emph{orbit} to capture this equivalence.
\subsection{Permutation Orbits}\label{subsec:permutation_orbits}
\begin{definition}
Let $\Pi_4$ denote the symmetric group of all $4!=24$ permutations
of the elements of $N_4$. An \emph{orbit} is a collection of nonsingular five-subset
bases that are equivalent under permutations $\pi \in \Pi_4$. Specifically, two bases
$
B=\{S_0,\ldots,S_4\}
\quad\text{and}\quad
B'=\{S'_0,\ldots,S'_4\}
$
belong to the same orbit if there exists a permutation $\pi$ of $N_4$
such that
$
B'=\{\pi(S_0),\ldots,\pi(S_4)\},
$
where $\pi(S)=\{\pi(i):i\in S\}$.
\end{definition}
\noindent For example, the bases $B=\{\emptyset,\{1\},\{1,2\},\{1,2,3\},N_4\}$ and $B'=\{\emptyset,\{2\},\{1,2\},\{1,2,3\},N_4\}$ are equivalent under the permutation $\pi\in\Pi_4$ defined by $\pi(1)=2,\ \pi(2)=1,\ \pi(3)=3,\ \pi(4)=4$, and hence belong to the same orbit.
The next lemma shows that the Bernstein coefficient system is invariant
under permutations in $\Pi_4$, so that a feasible coefficient system for
one basis induces a feasible coefficient system for every basis in the
same orbit.
\begin{lemma}[Permutation invariance of the Bernstein coefficient system]
\label{lem:symmetry_permutation_invariance}
Let $\mathcal{O}$ be an orbit of nonsingular five-subset bases, and let
$\Delta^{\mathcal{O}}$ denote a representative simplex of the orbit.
For any $\pi\in\Pi_4$, let $\Delta_{\pi}^{\mathcal{O}}$ denote the
simplex obtained from $\Delta^{\mathcal{O}}$ by applying $\pi$. Then the
Bernstein coefficient system
\eqref{eq:coeff_wise_27_equations_73_variables} is invariant under
permutation. In particular, if
$(\mb{p}^{ab},\mb{\gamma}^{ab},c^{ab})$ is a feasible solution
of the coefficient system for $\Delta^{\mathcal{O}}$, then the
appropriately permuted coefficients
$
\left(\mb{p}^{ab}_{\pi},\mb{\gamma}^{ab}_{\pi},c^{ab}\right)
$
provide a feasible solution for $\Delta_{\pi}^{\mathcal{O}}$.
\end{lemma}

\begin{proof}
    The proof is relegated to Appendix~\ref{proof:lem_symmetry_permutation_invariance}.
\end{proof}
\begin{corollary}[Reduction to permutation orbits]
\label{cor:permutation_orbits}
The $3,008$ nonsingular five-subset bases in
$\mathcal{B}_{\mathrm{NS}}$ partition into $183$ representative orbits under
permutations $\pi\in\Pi_4$, as shown in Table~\ref{tab:orbit_subdivision_statistics} of Appendix \ref{appendix_sec:orbit_representatives_and_subdivision_stats}. Hence, for each coefficient block $(a,b)$, by
Lemma~\ref{lem:symmetry_permutation_invariance}, it suffices to
establish the Bernstein coefficient system
\eqref{eq:coeff_wise_27_equations_73_variables} for one representative
basis from each of the $183$ orbits.
\end{corollary}

\begin{proof}
Two bases belong to the same orbit if one can be obtained from the other
by a permutation of the four ground set elements in $N_4$. By
Lemma~\ref{lem:symmetry_permutation_invariance}, a feasible coefficient
system for one basis induces a feasible coefficient system for every
basis in its orbit. Since the $3,008$ nonsingular bases partition into
$183$ such orbits, it suffices to verify one representative from each
orbit. With $15$ Bernstein coefficient blocks $(a,b)$, this amounts to solving
$15\times183=2,745$ linear coefficient systems of the form \eqref{eq:coeff_wise_27_equations_73_variables}.
\end{proof}
\noindent Having reduced the problem to one representative basis from each of the
$183$ permutation orbits, we next address the possibility that the
Bernstein coefficient system is infeasible on the entire representative
simplex $\Delta^\mathcal{O}$ for a given coefficient block $(a,b)$ by
providing a recursive subdivision strategy in Algorithm~\ref{alg:bernstein_subdivision}.
\subsection{Infeasibility and recursive subdivision algorithm}
\begin{algorithm}[H]
\caption{Recursive Simplex subdivision for Bernstein coefficient feasibility}
\label{alg:bernstein_subdivision}
\begin{algorithmic}[1]
\REQUIRE Representative simplex $\Delta_{\mathcal O}$, coefficient
block $(a,b)$, and maximum depth $D_{\max}=10$
\ENSURE A collection of leaf simplices covering $\Delta_{\mathcal O}$,
together with, for each feasible leaf, its depth and a certificate
$(\mb{p}^{ab},\mb{\gamma}^{ab},c^{ab})$ satisfying
\eqref{eq:coeff_wise_27_equations_73_variables}

\STATE Initialize a queue with the root simplex $\Delta_{\mathcal O}$ at depth $0$.
\STATE Initialize the set of feasible leaf certificates as empty.
\WHILE{the queue is nonempty}
    \STATE Remove a simplex $\Delta$ and its depth $d$ from the queue.
    \STATE Solve the coefficient system
    \eqref{eq:coeff_wise_27_equations_73_variables} for $\Delta$ and block
    $(a,b)$.
    \IF{the coefficient system is feasible}
        \STATE Store the leaf simplex $\Delta$, its depth $d$, and the
        certificate $(\mb{p}^{ab},\mb{\gamma}^{ab},c^{ab})$.
    \ELSIF{$d < D_{\max}$}
        \STATE Identify the longest edge of $\Delta$.
        \STATE Bisect the longest edge of $\Delta$ to obtain two child
        simplices $\Delta_1$ and $\Delta_2$ that differ in exactly one
        vertex.
        \STATE Add $\Delta_1$ and $\Delta_2$ to the queue with depth $d+1$.
    \ELSE
        \STATE Declare $\Delta$ unresolved.
    \ENDIF
\ENDWHILE
\RETURN The feasible leaf simplices, their depths, and the associated
certificates $(\mb{p}^{ab},\mb{\gamma}^{ab},c^{ab})$.
\end{algorithmic}
\end{algorithm}
Algorithm \ref{alg:bernstein_subdivision} describes a recursive subdivision strategy that splits an infeasible simplex using its longest
edge. A child simplex for which the coefficient system is
feasible is retained as a leaf, while an infeasible child is further
subdivided. The procedure terminates when all resulting leaf simplices
admit a feasible certificate, thereby providing a certificate over the
entire original simplex, since the collection of leaf simplices cover $\Delta_{\mathcal O}$. We note that each subdivision in Algorithm~\ref{alg:bernstein_subdivision}
bisects an edge and produces two child simplices, each obtained by replacing
one of the two endpoints of the bisected edge by its midpoint. Hence every
child simplex has vertices that are affine combinations of the vertices of
the root simplex $\Delta^{\mathcal O}$. Table~\ref{tab:representative_orbit_subdivision_statistics} presents six
representative permutation orbits spanning the ranges of the number of
terminating leaves and subdivision depth across the $183$ orbits. Orbits $31$ and $164$ terminate with the smallest number of leaves $4$
with a maximum depth $2$, orbits $83$ and $50$ terminate with the largest
number of leaves $89$ and $103$ with depth $10$, and orbits $61$ and $109$
represent intermediate cases near the median, both with depth $6$. The
subdivision statistics for the remaining orbits are reported in
Table~\ref{tab:orbit_subdivision_statistics} of Appendix \ref{appendix_sec:orbit_representatives_and_subdivision_stats}.
\begin{table}[htbp]
\centering
\caption{Representative permutation orbits illustrating the range of subdivision}
\label{tab:representative_orbit_subdivision_statistics}
\small\setlength{\tabcolsep}{4pt}
\renewcommand{\arraystretch}{1.1}
\begin{tabular}{c p{0.53\linewidth} c c}
\hline
Orbit & Representative basis & Leaves & Depth \\
\hline
31
& $\{\emptyset,\{1\},\{1,2\},\{1,2,3\},\{1,2,3,4\}\}$
& 4 & 2 \\[0.3em]

164
& $\{\{1,2\},\{1,3\},\{2,3\},\{1,2,3\},\{1,2,3,4\}\}$
& 4 & 2 \\[0.3em]

61
& $\{\emptyset,\{1,2,3\},\{1,2,4\},\{1,3,4\},\{1,2,3,4\}\}$
& 20 & 6 \\[0.3em]

109
& $\{\{1\},\{2\},\{1,2,3\},\{1,2,4\},\{1,2,3,4\}\}$
& 25 & 6 \\[0.3em]

83
& $\{\{1\},\{2\},\{3\},\{4\},\{1,2,3,4\}\}$
& 89 & 10 \\[0.3em]

50
& $\{\emptyset,\{1,2\},\{1,3\},\{1,4\},\{2,3,4\}\}$
& 103 & 10 \\
\hline
\end{tabular}
\end{table}
\subsection{Computational experiments}

We implemented Algorithm~\ref{alg:bernstein_subdivision} in Python 3.10.0 and
used linear programming to test the feasibility of the $2,745$ Bernstein
coefficient systems.
All computational experiments were conducted on a
MacBook Pro equipped with an Apple M4 Pro processor and 24~GB of
memory, running macOS Sequoia 15.2.

All $2,745$ coefficient systems were found to be feasible, with a total of $4476$ successful leaf simplices, which proves that the $4/3$ bound holds for $n=4$ as claimed in Theorem \ref{thm:n=4_main_result}. Table~\ref{tab:orbit_subdivision_statistics} reports the
corresponding subdivision statistics, including the number of terminal
leaf simplices and the maximum subdivision depth.
Complete output details, including a JSON file for each of the $183$ permutation orbits containing the terminal leaf simplices, subdivision depths, and Bernstein coefficients for each coefficient block, together with the Python implementation of the subdivision procedure, are available at the accompanying GitHub repository\footnote{\href{https://github.com/ArjunKRamachandra/CorrelationGap_4dimensions}{github.com/ArjunKRamachandra/CorrelationGap\_\_4dimensions}}. 

To illustrate the complete recursive subdivision procedure, we next consider a representative permutation orbit, derive cone certificates for the resulting leaf simplices, and reconstruct the corresponding numerator and denominator distributions from their Bernstein coefficients.
\subsection{Illustrative subdivision example}\label{subsec:illustrative_subdivision_example}
By the symmetry established in Lemma~\ref{lem:symmetry_permutation_invariance},
it is sufficient to construct certificates for each coeficient block $(a,b)$ using one representative
simplex from each orbit. To illustrate the subdivision procedure, we
consider the representative simplex
for Orbit
164  from Table \ref{tab:orbit_subdivision_statistics}, denoted by $\Delta^{\mathcal O}=
\operatorname{conv}\{\mb{v}^*_0,\mb{v}^*_1,\mb{v}^*_2,\mb{v}^*_3,\mb{v}^*_4\}$,
where
\[
\mb{v}^*_0=\mb{e}_1+\mb{e}_2,\quad
\mb{v}^*_1=\mb{e}_1+\mb{e}_3,\quad
\mb{v}^*_2=\mb{e}_2+\mb{e}_3,\quad
\mb{v}^*_3=\mb{e}_1+\mb{e}_2+\mb{e}_3,\quad
\mb{v}^*_4=\mb{1}.
\]
 are the incidence vectors
corresponding to the five-subset basis
$
B^\star
=
\left\{
\{1,2\},\{1,3\},\{2,3\},\{1,2,3\},N_4
\right\}$. Since
$
\mb{x}=\sum_{i=0}^4\eta_i\mb{v}^*_i$,
we have
\[
x_1=\eta_0+\eta_1+\eta_3+\eta_4,\qquad
x_2=\eta_0+\eta_2+\eta_3+\eta_4,\qquad
x_3=\eta_1+\eta_2+\eta_3+\eta_4,\qquad
x_4=\eta_4,
\]
Since $\sum_{i=0}^4\eta_i=1$, we obtain $
x_1=1-\eta_2,\;
x_2=1-\eta_1,\;
x_3=1-\eta_0,\;
x_4=\eta_4$, 
and
\[
\begin{aligned}
x_1+x_2+x_3-x_4
&=3-(\eta_0+\eta_1+\eta_2)-\eta_4\\
&=3-(1-\eta_3)\\
&=2+\eta_3
\ge2,
\end{aligned}
\]
where the last equality follows from $\eta_3\ge0$. Hence the marginal vector $\mb{x}$ whose optimal numerator
vertex $\mb{u}^\star(\mb{x},\mb{f})$ is supported on $B^\star$ must satisfy
$
x_1+x_2+x_3-x_4\ge2.
$ throughout the simplex $\Delta^{\mathcal O}$. 

The computational results in Table~\ref{tab:orbit_subdivision_statistics} show that the
representative simplex for Orbit~164 is certified by four leaf simplices,
with a maximum subdivision depth two. Since subdivision is performed
only when the coefficient system is infeasible, the root simplex and both
of its depth-one children are infeasible. Following the recursive
subdivision procedure in Algorithm~\ref{alg:bernstein_subdivision}, the
first bisection is performed along the longest failed edge
$(0,1)$ of the root simplex, while the subsequent bisections are along
the longest failed edge $(1,2)$ producing $\Delta_1,\Delta_2$ and edge $(0,2)$ producing $\Delta_3,\Delta_4 $ at
depth two. Table~\ref{tab:sub_simplex_vertices} lists the four leaf
simplices together with their vertices $\mb{v}_i^j$, where
$\mb{v}_i^j$ denotes the $i$th vertex of the $j$th leaf simplex and their convex hull representations in
terms of the representative vertices
$\mb{v}_0^\star,\ldots,\mb{v}_4^\star$. Each leaf-simplex vertex is either
a representative vertex or the midpoint of an edge of the representative
simplex.
\begin{table}[H]
\centering
\caption{Vertices of the root, intermediate, and leaf simplices for Orbit~164.}
\label{tab:sub_simplex_vertices}
\begin{tabular}{c|ccccc|l}
\hline
Simplex
& $\mb{v}_0^j$ & $\mb{v}_1^j$ & $\mb{v}_2^j$ & $\mb{v}_3^j$ & $\mb{v}_4^j$
& convex hull representation
\\
\hline




$\Delta_1$
&
$\left(1,\frac12,\frac12,0\right)$
&
$\left(\frac12,\frac12,1,0\right)$
&
$(0,1,1,0)$
&
$(1,1,1,0)$
&
$(1,1,1,1)$
&
$\operatorname{conv}\left\{
\frac{\mb{v}_0^\star+\mb{v}_1^\star}{2},
\frac{\mb{v}_1^\star+\mb{v}_2^\star}{2},
\mb{v}_2^\star,\mb{v}_3^\star,\mb{v}_4^\star
\right\}
$
\\[0.8em]

$\Delta_2$
&
$\left(1,\frac12,\frac12,0\right)$
&
$(1,0,1,0)$
&
$\left(\frac12,\frac12,1,0\right)$
&
$(1,1,1,0)$
&
$(1,1,1,1)$
&
$\operatorname{conv}\left\{
\frac{\mb{v}_0^\star+\mb{v}_1^\star}{2},
\mb{v}_1^\star,
\frac{\mb{v}_1^\star+\mb{v}_2^\star}{2},
\mb{v}_3^\star,\mb{v}_4^\star
\right\}
$
\\[0.8em]

$\Delta_3$
&
$\left(\frac12,1,\frac12,0\right)$
&
$\left(1,\frac12,\frac12,0\right)$
&
$(0,1,1,0)$
&
$(1,1,1,0)$
&
$(1,1,1,1)$
&
$\operatorname{conv}\left\{
\frac{\mb{v}_0^\star+\mb{v}_2^\star}{2},
\frac{\mb{v}_0^\star+\mb{v}_1^\star}{2},
\mb{v}_2^\star,\mb{v}_3^\star,\mb{v}_4^\star
\right\}
$
\\[0.8em]

$\Delta_4$
&
$(1,1,0,0)$
&
$\left(1,\frac12,\frac12,0\right)$
&
$\left(\frac12,1,\frac12,0\right)$
&
$(1,1,1,0)$
&
$(1,1,1,1)$
&
$\operatorname{conv}\left\{
\mb{v}_0^\star,
\frac{\mb{v}_0^\star+\mb{v}_1^\star}{2},
\frac{\mb{v}_0^\star+\mb{v}_2^\star}{2},
\mb{v}_3^\star,\mb{v}_4^\star
\right\}
$
\\
\hline
\end{tabular}
\end{table}

The four leaf simplices can alternatively be described through the convex
coordinates of $\mb{x}$ with respect to their five vertices. Table~\ref{tab:sub_simplex_convex_coefficients} gives the corresponding
convex coefficients $\eta^j_i$, where $\eta^j_i$ denotes the convex
coefficient of the $i$th vertex of the $j$th leaf simplex and should be
distinguished from the coefficients $\eta_i$ associated with the vertices
of the representative simplex $\Delta^{\mathcal O}$. These coefficients
are obtained from
$
\mb{x}=\sum_{i=0}^4\eta^j_i\mb{v}_i^j
$ for each $j=1,2,3,4$
and are affine functions of the marginal probabilities.
The marginal probability
conditions in the last column are obtained by requiring all five convex
coefficients to be nonnegative, together with the root-simplex conditions.
These restrictions characterize the four regions in which the corresponding
optimal numerator vertex $\mb{u}^\star(\mb{x},\mb{f})$ is supported on the
five vertices of the respective leaf simplex. Each subdivision corresponds
to slicing the current parent simplex by a hyperplane. 
\begin{table}[htbp]
\centering
\caption{Convex coefficients and marginal conditions for the four
feasible leaf simplices of $\Delta^{\mathcal O}$}
\label{tab:sub_simplex_convex_coefficients}
\small
\setlength{\tabcolsep}{3pt}
\renewcommand{\arraystretch}{1.05}

\begin{tabular}{|c|c|c|c|c|c|c|}
\hline
\rule{0pt}{2.7ex}sub-simplex
& $\eta^j_0$ & $\eta^j_1$ & $\eta^j_2$ & $\eta^j_3$ & $\eta^j_4$
& Marginal conditions
\\
\hline
$\Delta_1$
&
$2(1-x_3)$
&
$2(x_3-x_2)$
&
$1-x_1+x_2-x_3$
&
$x_1+x_2+x_3-x_4-2$
&
$x_4$
&
$\begin{array}{c}
x_2\le x_3,\\
x_1+x_3\le1+x_2
\end{array}$
\\[1.0em]
\hline

$\Delta_2$
&
$2(1-x_3)$
&
$x_1-x_2+x_3-1$
&
$2(1-x_1)$
&
$x_1+x_2+x_3-x_4-2$
&
$x_4$
&
$\begin{array}{c}
x_2\le x_3,\\
x_1+x_3\ge1+x_2
\end{array}$
\\[1.0em]
\hline

$\Delta_3$
&
$2(x_2-x_3)$
&
$2(1-x_2)$
&
$1-x_1-x_2+x_3$
&
$x_1+x_2+x_3-x_4-2$
&
$x_4$
&
$\begin{array}{c}
x_2\ge x_3,\\
x_1+x_2\le1+x_3
\end{array}$
\\[1.0em]
\hline

$\Delta_4$
&
$x_1+x_2-x_3-1$
&
$2(1-x_2)$
&
$2(1-x_1)$
&
$x_1+x_2+x_3-x_4-2$
&
$x_4$
&
$\begin{array}{c}
x_2\ge x_3,\\
x_1+x_2\ge1+x_3
\end{array}$
\\

\hline
\end{tabular}
\end{table}
 Specifically, the
original representative simplex is first split by $x_2=x_3$. The branch
$x_2\le x_3$ is then split by $x_1+x_3=1+x_2$, yielding $\Delta_1$ and
$\Delta_2$, while the branch $x_2\ge x_3$ is split by
$x_1+x_2=1+x_3$, yielding $\Delta_3$ and $\Delta_4$. Equivalently, each
leaf is the intersection of the representative simplex with the two
corresponding half spaces. These four regions are mutually exclusive and
cover $\Delta^{\mathcal O}$ as shown in Figure \ref{fig:representative_simplex_subdivision}. Hence, every $\mb{x}\in\Delta^{\mathcal O}$
belongs to one of the four leaf simplices, and the corresponding Bernstein
certificates suffice to certify the entire representative simplex.
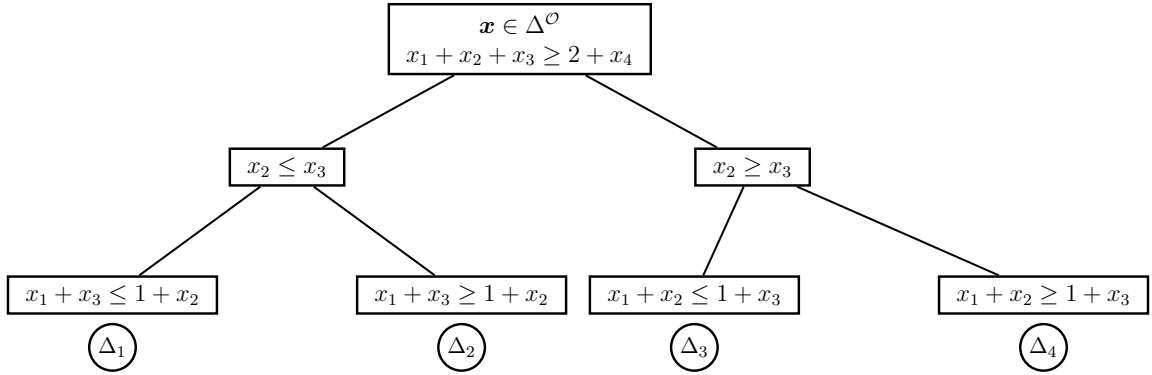
\begin{figure}[H]
\centering
\resizebox{0.85\textwidth}{!}{%
\begin{tikzpicture}[
    >=Stealth,
    every path/.style={line width=1.0pt},
    every node/.style={font=\bfseries\large}
]

\tikzset{
  box/.style={
    draw,
    rectangle,
    very thick,
    align=center,
    minimum width=0pt,
    minimum height=0pt,
    inner xsep=3mm,
    inner ysep=1.5mm
  },
  leaf/.style={
    draw,
    circle,
    very thick,
    minimum size=8mm,
    inner sep=0pt,
    font=\bfseries
  }
}

\node[box] (root) at (0,0)
{$\mb{x}\in\Delta^{\mathcal O}$\\
$x_1+x_2+x_3\ge 2+x_4$};

\node[box] (left) at (-4,-2.2)
{$x_2\le x_3$};

\node[box] (right) at (4,-2.2)
{$x_2\ge x_3$};

\draw (root) -- (left);
\draw (root) -- (right);

\node[box] (leftlow) at (-7,-4.4)
{$x_1+x_3\le1+x_2$};

\node[box] (lefthigh) at (-1,-4.4)
{$x_1+x_3\ge1+x_2$};

\draw (left) -- (leftlow);
\draw (left) -- (lefthigh);

\node[leaf] (L1) at (-7,-5.3) {$\Delta_1$};
\node[leaf] (L2) at (-1,-5.3) {$\Delta_2$};

\node[box] (rightlow) at (3,-4.4)
{$x_1+x_2\le1+x_3$};

\node[box] (righthigh) at (9,-4.4)
{$x_1+x_2\ge1+x_3$};

\draw (right) -- (rightlow);
\draw (right) -- (righthigh);

\node[leaf] (L3) at (3,-5.3) {$\Delta_3$};
\node[leaf] (L4) at (9,-5.3) {$\Delta_4$};

\end{tikzpicture}%
}
\caption{Subdivision of the representative simplex $\Delta^{\mathcal O}$
for Orbit~164 into four certified leaf simplices.}
\label{fig:representative_simplex_subdivision}
\end{figure}

To illustrate the construction of the cone certificate, we next highlight one of the
leaf sub-simplices, namely sub-simplex~4. Table~\ref{tab:our_leaf_4} reports, for each block $(a,b)$, the
Bernstein coefficient vectors $\mb{p}^{ab}$ and $\mb{\gamma}^{ab}$,
together with the associated constant $c^{ab}$ for sub-simplex~4. The corresponding Bernstein coefficient tables for the other three sub-simplices are provided in
Appendix~\ref{appendix_sec:sample_bernstein_tables}. Table~\ref{tab:leaf4_distributions} illustrates the reconstructed optimal
numerator and feasible denominator distributions for sub-simplex~4 in terms
of the marginal probabilities. The numerator distribution is optimal whenever
$\mb{u}^*(\mb{x},\mb{f})=\argmax_{\mb{u}\in\mathcal{F}_{\mathrm{U}}(\mb{x})}\mb{f}^{\top}\mb{u}$
is supported on
$B^\star=\left\{\{1,2\},\{1,3\},\{2,3\},\{1,2,3\},N_4\right\}$
for a given monotone submodular function vector $\mb{f}$ and a marginal probability vector $\mb{x} \in [0,1]^4$ satisfying
\begin{align}\label{al:marginals_sub_simplex_4}
x_1+x_2+x_3\ge 2+x_4,\quad
x_2\ge x_3,\quad
x_1+x_2\ge 1+x_3.
\end{align}
It is obtained directly from
\[
\mb{u}^\star(\mb{x})=\sum_{i=0}^4\eta^4_i(\mb{x})\mb{\rho}_i,
\qquad
\mb{\rho}_i=\left(\mb{e}_{S_0},\frac{\mb{e}_{S_0}+\mb{e}_{S_1}}{2},\frac{\mb{e}_{S_0}+\mb{e}_{S_2}}{2},\mb{e}_{S_3},\mb{e}_{S_4}\right)_i,
\]
where $\mb{e}_{S_i}$ is the $16$-dimensional unit vector corresponding to the incidence vector of the $i$th vertex of the original representative simplex $\Delta^{\mathcal O}$, and the convex coefficients $\eta^4_i(\mb{x})$ are obtained from the last row of Table~\ref{tab:sub_simplex_convex_coefficients}. 

\begin{table}[H]
\centering
\caption{Bernstein coefficients for sub-simplex 4.}
\label{tab:our_leaf_4}
\footnotesize
\setlength{\tabcolsep}{2pt}
\renewcommand{\arraystretch}{0.82}

\begin{tabularx}{\textwidth}{
>{\raggedright\arraybackslash}p{0.08\textwidth}
>{\raggedright\arraybackslash}p{0.25\textwidth}
>{\raggedright\arraybackslash}X
>{\centering\arraybackslash}p{0.07\textwidth}}
\toprule
Block & $\mb{p}^{ab}$ & $\mb{\gamma}^{ab}$ & $c^{ab}$ \\
\midrule
$(0,0)$ & $\{\,3:1\,\}$ & $\{\,0:1/3\; 4:1/3\,\}$ & $1/3$ \\
$(0,1)$ & $\{\,1:1/4\; 3:1/2\; 7:1/4\,\}$ & $\{\,0:1/3\; 10:1/12\; 15:1/4\,\}$ & $1/3$ \\
$(0,2)$ & $\{\,2:1/4\; 3:1/2\; 7:1/4\,\}$ & $\{\,1:1/3\; 10:1/12\; 17:1/4\,\}$ & $1/3$ \\
$(0,3)$ & $\{\,3:1/2\; 7:1/2\,\}$ & $\{\,0:1/3\; 4:1/3\; 10:1/6\,\}$ & $1/3$ \\
$(0,4)$ & $\{\,7:1/2\; 11:1/2\,\}$ & $\{\,0:1/3\; 4:1/3\; 10:1/6\; 11:1/6\; 44:1/2\,\}$ & $1/3$ \\
$(1,1)$ & $\{\,1:1/4\; 3:1/4\; 5:1/4\; 7:1/4\,\}$ & $\{\,0:1/3\; 10:1/6\; 15:1/6\,\}$ & $1/3$ \\
$(1,2)$ & $\{\,0:1/8\; 3:3/8\; 5:1/8\; 6:1/8\; 7:1/4\,\}$ & $\{\,0:1/6\; 4:1/6\; 10:1/6\; 15:1/12\; 17:1/12\,\}$ & $1/3$ \\
$(1,3)$ & $\{\,3:1/4\; 5:1/4\; 7:1/2\,\}$ & $\{\,0:1/3\; 4:1/4\; 5:1/12\; 10:1/6\,\}$ & $1/3$ \\
$(1,4)$ & $\{\,7:1/2\; 11:1/4\; 13:1/4\,\}$ & $\{\,1:1/12\; 2:1/4\; 14:1/12\; 15:1/6\; 16:1/12\; 17:1/4\; 25:1/12\; 43:1/12\; 44:1/4\; 47:1/6\; 55:1/4\,\}$ & $1/3$ \\
$(2,2)$ & $\{\,2:1/4\; 3:1/4\; 6:1/4\; 7:1/4\,\}$ & $\{\,1:1/3\; 10:1/6\; 17:1/6\,\}$ & $1/3$ \\
$(2,3)$ & $\{\,3:1/4\; 6:1/4\; 7:1/2\,\}$ & $\{\,0:1/4\; 1:1/12\; 4:1/4\; 8:1/12\; 10:1/6\,\}$ & $1/3$ \\
$(2,4)$ & $\{\,7:1/2\; 11:1/4\; 14:1/4\,\}$ & $\{\,0:1/12\; 2:1/12\; 14:1/12\; 15:1/4\; 17:1/6\; 18:1/12\; 22:1/6\; 23:1/12\; 34:1/6\; 40:1/4\; 44:1/4\; 55:1/4\,\}$ & $1/3$ \\
$(3,3)$ & $\{\,7:1\,\}$ & $\{\,0:1/3\; 4:1/3\; 10:1/3\,\}$ & $1/3$ \\
$(3,4)$ & $\{\,7:1/2\; 15:1/2\,\}$ & $\{\,0:1/3\; 4:1/3\; 10:1/3\; 19:1/6\,\}$ & $1/3$ \\
$(4,4)$ & $\{\,15:1\,\}$ & $\{\,0:1/3\; 4:1/3\; 10:1/3\; 19:1/3\,\}$ & $1/3$ \\
\bottomrule
\end{tabularx}
\end{table}
The pairwise distribution, on the other hand, is not necessarily optimal for the given $\mb{x}$ and is reconstructed from
\[
\mb{p}(\mb{x})=\sum_{0\le a\le b\le4}\mb{p}^{ab}B_{ab}(\mb{\eta}(\mb{x})),
\]
using the Bernstein coefficient distributions $\mb{p}^{ab}$ reported in Table~\ref{tab:our_leaf_4} for each block $(a,b)$. By Lemma \ref{lem:cone_certificate}, the distributions in Table \ref{tab:leaf4_distributions} ensure that the cone certificate equation in \eqref{eq:cone_certificate} is satisfied for any monotone submodular function vector $\mb{f}$ and any $\mb{x}$ satisfying \eqref{al:marginals_sub_simplex_4}. It can be observed that the numerator distribution is supported on five points, precisely corresponding to the five vertices of the sub-simplex, while the pairwise distribution is supported on eleven points.
\begin{table}[H]
\centering
\caption{Optimal numerator and feasible denominator distributions for sub-simplex~4.}
\label{tab:leaf4_distributions}
\begin{tabular}{|c|c@{\qquad}|c|c|}
\toprule
Subset $S$ & $\mb{u}^*_S(\mb{x})$ & Subset $S$ & $\mb{p}_S(\mb{x})$ \\
\midrule

$\{1,2\}$
&
$\displaystyle 1-x_3$
&
$\emptyset$
&
$\displaystyle
(1-x_1)(1-x_2)
$
\\[1em]

$\{1,3\}$
&
$\displaystyle 1-x_2$
&
$\{1\}$
&
$\displaystyle
(x_1-x_3)(1-x_2)
$
\\[1em]

$\{2,3\}$
&
$\displaystyle 1-x_1$
&
$\{2\}$
&
$\displaystyle
(1-x_1)(x_2-x_3)
$
\\[1em]

$\{1,2,3\}$
&
$\displaystyle x_1+x_2+x_3-x_4-2$
&
$\{1,2\}$
&
$\displaystyle
x_1(x_2-x_3)+x_3(1-x_2)-x_4(1-x_3)
$
\\[1em]

$N_4$
&
$\displaystyle x_4$
&
$\{1,3\}$
&
$\displaystyle
(1-x_2)(x_3-x_4)
$
\\[1em]

&
&
$\{2,3\}$
&
$\displaystyle
(1-x_1)(x_3-x_4)
$
\\[1em]

&
&
$\{1,2,3\}$
&
$\displaystyle
(x_1+x_2-1)x_3-x_4(x_1+x_2+x_3-2)
$
\\[1em]

&
&
$\{1,2,4\}$
&
$\displaystyle
x_4(1-x_3)
$
\\[1em]

&
&
$\{1,3,4\}$
&
$\displaystyle
x_4(1-x_2)
$
\\[1em]

&
&
$\{2,3,4\}$
&
$\displaystyle
x_4(1-x_1)
$
\\[1em]

&
&
$N_4$
&
$\displaystyle
x_4(x_1+x_2+x_3-2)
$
\\
\bottomrule
\end{tabular}
\end{table}

\section{Worst case asymptotic pairwise independent correlation gap}\label{sec:asymptotic_pw_correlation_gap}
In this section, we prove that the worst case
pairwise independent correlation gap  attains
$e/(e-1)$ asymptotically as $n \to \infty$, the same worst case bound established in \cite{Agrawal2012} under mutual independence. As an immediate consequence, since every $t$-wise independent distribution ($t\ge 2$) is necessarily pairwise independent, our result implies that the worst case $t$-wise independent correlation gap is also exactly $e/(e-1)$. We first define a specific class of coverage functions that are monotone and submodular by definition that will be used to prove the main result. 
\subsection*{Coverage functions}
Let $\mathcal H_m=\{H\subseteq 2^N: |H|=m\}$ be a family of $m$-subsets of $N$, called
\emph{features}. Each feature $H\in\mathcal H$ is assigned a nonnegative
weight $w_H\ge0$. Define a corresponding nonnegative set function $g_m: 2^{N} \rightarrow \mathbb{R}_+$ as:
\[
g_m(S)
=
\sum_{H\in\mathcal H_m}
w_H\,
\mb 1\{S\cap H\neq\varnothing\},
\qquad S\subseteq N.
\]
The function $g_m$ is a weighted coverage function and for a given $S$, equals the total weight of the features covered by $S$. Such a weighted coverage function is known to be nonnegative, nondecreasing and submodular (see \cite{bachsub}). 

\begin{theorem}\label{thm:main_result_asymptotic}[Asymptotically tight bound]
For any $n$, any nonnegative monotone submodular function $f: 2^{N} \rightarrow \mathbb{R}_+$ and any $\mb{x} \in [0,1]^n$: 
$$\frac{f^{+}(\mb{x})}{f^{++}(\mb{x})} \leq \frac{e}{e-1}$$
Moreover, the bound is asymptotically attained when:
\begin{enumerate}[label=\roman*)]
\item
The ground set is partitioned into $m$ blocks as
$
N=A_1\cup A_2\cup\cdots\cup A_{m-1}\cup A_m,
$ where \newline
$$|A_r|=\ell=\left\lfloor\dfrac{n}{m}\right\rfloor,\; \forall r\in M \setminus \{m\},\quad
|A_m|=\ell_{m}=n-(m-1)\ell,$$
and the feature family is defined as
$
\mathcal H_m=A_1\times A_2\times\cdots\times A_{m-1}\times A_m$, 
consisting of all $m$-tuples obtained by selecting one element from each set of 
the partition.
\item
The asymptotic regime considers the number of partition blocks $m$ as a
function of $n$, denoted by $m=m(n)$, with both $n$ and $m(n)$ tending to
infinity and $m(n)$ growing sublinearly in $n$, \emph{i.e.},
\[
n\to\infty,\qquad m(n)\to\infty,\qquad m(n)=o(n).
\]
\item
The marginal probabilities are set to be identical with
$x_i=\frac1m,\; \forall i\in N$.
\item
The function $
f(S)$ is the coverage function $g_m(S)$ with unit feature weights $w_H=1$ for every $H\in\mathcal H_m$ and counts the number of features covered by at least one element of $S$.
\end{enumerate}
\end{theorem}
\begin{proof}
We describe the key idea first. Let $\mb{x}^{\mathrm{id}}=\left(\frac1m,\frac1m,\ldots,\frac1m\right)$ denote the identical marginal probability vector. For any upper bound
$
\overline{f}^{++}(\mb{x}^{\mathrm{id}})
\ge
f^{++}(\mb{x}^{\mathrm{id}})
$,
we have
\[
\frac{f^{+}(\mb{x}^{\mathrm{id}})}
{\overline{f}^{++}(\mb{x}^{\mathrm{id}})}
\le
\frac{f^{+}(\mb{x}^{\mathrm{id}})}
{f^{++}(\mb{x}^{\mathrm{id}})}.
\]
Since~\cite{Agrawal2012} proved that
$$
\frac{f^{+}(\mb{x})}{f^{++}(\mb{x})}\le\frac{e}{e-1},
$$
for every marginal probability vector $\mb{x}$, it is sufficient to construct an upper bound
\(\overline{f}^{++}(\mb{x}^{\mathrm{id}})\) satisfying
\begin{equation}
\label{eq:sufficient_ratio}
\frac{f^{+}(\mb{x}^{\mathrm{id}})}
{\overline{f}^{++}(\mb{x}^{\mathrm{id}})}
\longrightarrow
\frac{e}{e-1},
\end{equation}
under the conditions $(i)-(iv)$, forcing both ratios to converge to the common limit $e/(e-1)$.
   For the rest of the proof, we proceed step by step, first establishing the necessary preliminaries, beginning with the setup and definitions in Section~\ref{subsec:setup_and_definitions}. This is followed by the optimal concave closure value with identical marginals in Section~\ref{subsec:concave_closure_identical_marginals}, and an upper bound on the pairwise extension with identical marginals through a scaled asymptotic reduced dual formulation in Section~\ref{subsec:pairwise_scaled_asymptotic_dual}, culminating in the detailed proof 
   in Section~\ref{sec:thm_main_result_asymptotic}.
\end{proof}

\begin{remark}
It was proved in~\cite{ramachandra_special_cases_OR_letters} that the correlation gap is upper bounded by $4/3$ for identical marginal probabilities satisfying
\[
x\le \frac{1}{n-1}
\quad\text{or}\quad
x\ge \frac{n-2}{n-1}.
\]
In the present construction, $
x_i=\frac1m$, where \(m=m(n)\) satisfies
$
m(n)=o(n),\; m(n)\to\infty
\mbox{ as } n\to\infty.$
Consequently, for all sufficiently large \(n\),
$
2\le m \le n-2
$ and therefore
\[
\frac1{n-1}<\frac1m<\frac{n-2}{n-1}.
\]
Therefore, the regime considered in Theorem~\ref{thm:main_result_asymptotic} lies strictly outside the parameter ranges covered in \cite{ramachandra_special_cases_OR_letters}, and leads to a different upper bound of $e/(e-1)$. 
\end{remark}

\begin{remark}
Under the assumption $
m=m(n)\to\infty,\;
m=o(n)$,
both the number of partition blocks and the size of each partition block diverge as
\(n\to\infty\). Indeed,
\begin{align}\label{al:T_and_Q_size_infty}
m\to\infty,
\quad
\ell=\left\lfloor\frac{n}{m}\right\rfloor\to\infty,
\quad
\ell_{m}=n-(m-1)\ell\ge\frac{n}{m}\to\infty
\quad \mbox{   as   } n\to\infty
\end{align}
\end{remark}

\begin{figure}[H]
\centering
\begin{tikzpicture}[x=1.25cm,y=1cm]
\draw (0,0) rectangle ++(1.05,0.7);
\draw (1.05,0) rectangle ++(1.05,0.7);
\draw (2.10,0) rectangle ++(1.05,0.7);
\draw (3.15,0) rectangle ++(1.05,0.7);

\draw (4.75,0) rectangle ++(1.35,0.7);

\node at (0.525,0.35) {$A_1$};
\node at (1.575,0.35) {$A_2$};
\node at (2.625,0.35) {$\cdots$};
\node at (3.675,0.35) {$A_{m-1}$};
\node at (5.425,0.35) {$A_m$};

\draw[decorate,decoration={brace,amplitude=5pt}]
    (0,1.02) -- (6.10,1.02)
    node[midway,yshift=0.45cm] {$n$};

\draw[decorate,decoration={brace,amplitude=5pt,mirror}]
    (0,-0.12) -- (4.20,-0.12)
    node[midway,yshift=-0.42cm]
    {$m-1\text{ blocks},\quad |A_r|=\ell$};

\draw[decorate,decoration={brace,amplitude=5pt,mirror}]
    (4.75,-0.12) -- (6.10,-0.12)
    node[midway,yshift=-0.42cm]
    {$|A_m|=n-(m-1)\ell$};

\node at (3.05,2.05)
    {$n\to\infty,\qquad m\to\infty,\qquad \ell\to\infty$};
\end{tikzpicture}
\caption{Schematic representation of the partition of the ground set $N$
into $m-1$ blocks of cardinality $\ell$ and a final block $A_m$ of
cardinality $n-(m-1)\ell$ in the asymptotic regime.}
\label{fig:asymptotic_partition}
\end{figure}
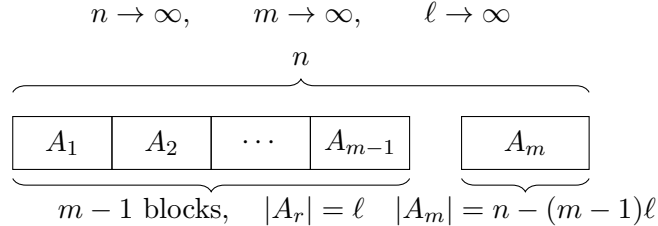
The next result shows that the asymptotic result in Theorem \ref{thm:main_result_asymptotic} immediately extends to $t$-wise independent random elements
($t\ge2$), since $t$-wise independence implies pairwise independence.
\begin{corollary}[Extension to $t$-wise independence]
Let $f^{(t)}(\mb{x})$, for $t>2$, denote the optimal value of the
$t$-wise independence formulation obtained by replacing the pairwise
independence constraints in \eqref{eq:pairwise_concave_closure} with
\[
\sum_{S\supseteq\{i_1,\ldots,i_r\}} p_S
=
\prod_{j=1}^r x_{i_j},
\qquad
\forall\, i_1<\cdots<i_r,\quad i_1,\ldots,i_r\in N,\quad
\forall\, 1\le r\le t.\] Then
\[
\frac{f^+(\mb x)}{f^{(t)}(\mb x)}
\le
\frac{e}{e-1}.
\] and the bound is attained under the the exact same conditions $(i)-(iv)$ in Theorem \ref{thm:main_result_asymptotic}.
\end{corollary}

\begin{proof}
Every $t$-wise independent distribution is, in particular, pairwise
independent. Hence the feasible region defining $f^{(t)}$ is contained
in that defining $f^{++}$ in formulation \eqref{eq:pairwise_concave_closure}, implying $
f^{(t)}(\mb{x}^{\mathrm{id}})\le f^{++}(\mb{x}^{\mathrm{id}})$. Therefore,
\[
\frac{f^+(\mb{x}^{\mathrm{id}})}{f^{++}(\mb{x}^{\mathrm{id}})}
\le
\frac{f^+(\mb{x}^{\mathrm{id}})}{f^{(t)}(\mb{x}^{\mathrm{id}})} \le \frac{f^+(\mb{x}^{\mathrm{id}})}{f^{(n)}(\mb{x}^{\mathrm{id}})}\le
\frac{e}{e-1},
\]
where the last inequality follows from the $e/(e-1)$ upper bound with mutual independence in \cite{Agrawal2012}. The result then follows from Theorem \ref{thm:main_result_asymptotic}, since $
\frac{f^+(\mb{x}^{\mathrm{id}})}{f^{++}(\mb{x}^{\mathrm{id}})}$ attains the bound
$e/(e-1)$ asymptotically, forcing all ratios to converge to the common limit $e/(e-1)$.
\end{proof}
\subsection{Preliminaries leading to proof of Theorem \ref{thm:main_result_asymptotic}}\label{subsec:preliminaries_main_theorem}
We now proceed to prove tightness of the $e/(e-1)$ bound in Theorem \ref{thm:main_result_asymptotic}. In preparation, we next reiterate the tightness conditions $(i)-(iv)$ in Theroem \ref{thm:main_result_asymptotic} along with some additional definitions needed to set up the tight instance.
\subsubsection{Setup and definitions}\label{subsec:setup_and_definitions}
In line with the tightness conditions $(i)-(iv)$ in Theroem \ref{thm:main_result_asymptotic}, we setup the problem instance as follows: 
\begin{enumerate}[label=\roman*)]
\item
The ground set is partitioned into $m$ blocks as $N=A_1\cup A_2\cup\cdots\cup A_{m-1}\cup A_m$ 
where the first $m-1$ blocks 
have equal cardinality $ 
|A_r|=\ell=\lfloor \frac nm\rfloor,\; r \in M \setminus \{m\},
$  and the last block consisting of the remaining elements has cardinality $|A_m|=\ell_{m}=n-(m-1)\ell$. 
\item 
Further, let $M=\{1,2,\ldots,m\}$. Then for every subset \(S\subseteq N\), define the block count $
a_r=|S\cap A_r|,\;r \in M$ as the number of elements from block $A_r$ contained in S. Then
$
\sum_{r=1}^m a_r=|S|,
$
since $\{A_r:r\in M\}$  partitions the ground set $N$.
\item 
The feature family $
\mathcal H_m=A_1\times A_2\times\cdots\times A_{m-1}\times A_m$, 
consists of all $m$-tuples obtained by selecting one element from each set of 
the partition.
\item
 Define the set function $f: 2^{N} \rightarrow \mathbb{R}_+$ as the monotone submodular union coverage function
\begin{equation}\label{eq:union_coverage_set_function_formula}
\begin{array}{ll}
f(S)&=g_m(S)\vspace{0.2cm}\\
&=
\sum_{H\in\mathcal H_m}
\mb 1\{S\cap H\neq\varnothing\},\vspace{0.2cm} \\
 &=\ell^{m-1}\ell_{m}-\prod_{r=1}^{m-1}(\ell-a_r)(\ell_{m}-a_m).
\end{array}
\end{equation}
for each $S \subseteq N$ where $\ell^{m-1}\ell_{m}$ is the total number of features in the feature family $
\mathcal H_m=A_1\times A_2\times\cdots\times A_{m-1}\times A_m$ and the subtracted term  
$
\prod_{r=1}^{m-1}(\ell-a_r)(\ell_{m}-a_m)$ represents the number of uncovered features. Thus, $f(S)$ counts the total number of features covered by at least one element of $S$.
\item For a subset \(S\subseteq N\), define  $\mb{u}(S)=(u_1(S),\ldots,u_m(S)) \in [0,1]^m$ where 
$$
u_r(S)=\dfrac{a_r}{\ell},\; r\in M \setminus \{m\},\;
u_m(S)=\dfrac{a_m}{\ell_{m}}$$ 
denote the normalized block counts of \(S\), and satisfy $
\ell\sum_{r=1}^{m-1}u_r(S)+Qu_m(S)=|S|
$. Then the set function $f(S)$ in \eqref{eq:union_coverage_set_function_formula}
can be transformed into a normalized coverage polynomial $F: [0,1]^m \rightarrow [0,1]$ as
\begin{align}\label{al:normalized_set_function}
    F(\mb{u}(S))=\frac{f(S)}{\ell^{m-1}\ell_{m}}
    =
    1-\prod_{r=1}^m(1-u_r(S)), \quad \forall S \subseteq N
\end{align}
We assume that all expectations in the definitions of $f^+(\mb{x})$ and $f^{++}(\mb{x})$ in \eqref{eq:concave_closure} and \eqref{eq:pairwise_concave_closure} are henceforth 
computed with respect to $F(\mb{u}(S))$ instead
of $f(S)$. Additionally, for notational convenience, we drop the dependence on $S$ 
 (unless necessary in the context of a statement) 
 and write $F(u)=
1-\prod_{r=1}^m(1-u_r)$.
\item Set the marginal probabilities to
$x_i=\dfrac1m$ for every $i\in N$.
\end{enumerate}
 We next derive the optimal solution with identical marginals to the concave closure in \eqref{eq:concave_closure} by identifying a worst case distribution. 
\subsubsection{Concave closure with identical probabilities and normalized coverage function}\label{subsec:concave_closure_identical_marginals}
 Since the normalized block counts satisfy $u_r(A_r)=1$ for every $r \in M$, the normalized coverage function defined in \eqref{al:normalized_set_function} satisfies $F(u(A_r))=1$ for every $r \in M$. Consequently, the concave closure in \eqref{eq:concave_closure} satisfies $f^{+}(\mb{x}^{\mathrm{id}})=1$, attained by the distribution
$$
\mb{p}^*(S)=
\begin{cases}
\dfrac1m, & S\in\{A_1,A_2,\ldots,A_m\},\\[0.2cm]
0, & \text{otherwise}.
\end{cases}
$$
Thus, it suffices to show that
\[
f^{++}(\mb{x}^{\mathrm{id}})
\longrightarrow
1-\frac1e
\]
as \(n\to\infty\) with \(m=m(n)\to\infty\) and \(m=o(n)\), so that
$$
\frac{f^{+}(\mb{x}^{\mathrm{id}})}
     {f^{++}(\mb{x}^{\mathrm{id}})}
\longrightarrow
\frac{e}{e-1}.
$$
In the next section, we construct an asymptotic reduced dual formulation with identical marginals whose optimal value upper bounds $f^{++}(\mb{x}^{\mathrm{id}})$ and asymptotically converges to $1-\frac1e$.

\subsubsection{Pairwise independent asymptotic reduced dual formulation with identical marginals}\label{subsec:pairwise_scaled_asymptotic_dual}

The dual of the pairwise independent extension in
\eqref{eq:pairwise_concave_closure} with the normalized submodular function in \eqref{al:normalized_set_function} can be written as:

\begin{equation}
\begin{array}{ll}
\displaystyle
f^{++}(\mb{x})
=
\min
&
\displaystyle
\lambda_0
+
\sum_{i\in N}\lambda_i x_i
+
\sum_{i<j}\lambda_{ij}x_ix_j
\\[2ex]
\mathrm{s.t.}
&
\displaystyle
\lambda_0
+
\sum_{i\in S}\lambda_i
+
\sum_{\{i,j\}\subseteq S}\lambda_{ij}
\ge
F(u(S)),
\qquad
\forall S\subseteq N.
\end{array}
\label{eq:pairwise_dual}
\end{equation}
where all dual variables are unrestricted due to the equality constraints in \eqref{eq:pairwise_concave_closure}. We next derive a reduced dual formulation of \eqref{eq:pairwise_dual} for identical marginals $\mb{x}=\mb{x}^{\mathrm{id}}$ by imposing symmetry restrictions on the dual variables $\lambda_i,\lambda_{ij}$, followed by an asymptotic reduction after appropriate scaling of the dual variables .
\begin{lemma}[Reduced dual formulation]
\label{lem:reduced_dual}
 The dual formulation in \eqref{eq:pairwise_dual} reduces to
\begin{equation}\label{eq:reduced_dual}
\begin{array}{ll}
\displaystyle
f^{++}_{\mathrm{RD}}(\mb{x}^{\mathrm{id}})
=
\underset{\footnotesize A,B,C,D}{\scalebox{1.25}{$\min$}}
&
\displaystyle
A
+B\frac{n}{m}
+\frac{D}{m^{2}}
\left((m-1)\binom{\ell}{2}+\binom{\ell_{m}}{2}\right)
+\frac{E}{m^{2}}
\left(\binom{m-1}{2}\ell^{2}+(m-1)\ell\ell_{m}\right)
\\[3ex]
\mathrm{s.t.}
&
\displaystyle
A
+B\sum_{r=1}^{m}a_r
+D\sum_{r=1}^{m}\binom{a_r}{2}
+E\sum_{1\le r<s\le m}a_ra_s
\ge
F(u(S)),
\qquad
\forall S\subseteq N.
\end{array}
\end{equation}
which provides an upper bound on $f^{++}(\mb{x}^{\mathrm{id}})$ \emph{i.e.},
\[
f^{++}(\mb{x}^{\mathrm{id}})
\le
f^{++}_{\mathrm{RD}}(\mb{x}^{\mathrm{id}}).
\]
\end{lemma}

\begin{proof}
Let $\lambda_0=A$ and impose the following symmetry restrictions on the remaining dual variables:
\begin{align}
\label{al:symmetry_dual}
\lambda_i=B,\quad \forall i\in N,
\qquad
\lambda_{ij}=
\begin{cases}
D,& i,j\in A_r,\ \text{for some }r\in M,\\
E,& i\in A_r,\ j\in A_s,\ r\neq s.
\end{cases}
\end{align}
Substituting these restricted dual variables into the objective of
\eqref{eq:pairwise_dual} along with $x_i=\frac1m$ for each $i\in N$, we obtain
\[
\lambda_0
+
\sum_{i\in N}\lambda_i x_i
+
\sum_{i<j}\lambda_{ij}x_ix_j
=
A+\frac{Bn}{m}
+\frac{D}{m^{2}}
\left((m-1)\binom{\ell}{2}+\binom{\ell_{m}}{2}\right)
+\frac{E}{m^{2}}
\left(\binom{m-1}{2}\ell^{2}+(m-1)\ell\ell_{m}\right)
\]
since there are
$(m-1)\binom{\ell}{2}+\binom{\ell_{m}}{2}$ pairs of elements belonging to the same partition block and
$\binom{m-1}{2}\ell^{2}+(m-1)\ell\ell_{m}$ pairs belonging to different partition blocks.

Similarly, for the constraints, let $S \subseteq N$ be an arbitrary subset. Recalling that
$a_r=|S\cap A_r|$ and $\{A_r:r\in M\}$ forms a partition of $N$, we have
\[
\sum_{i\in S}\lambda_i
=
B\sum_{r=1}^{m}a_r,
\]
and
\[
\sum_{\{i,j\}\subseteq S}\lambda_{ij}
=
D\sum_{r=1}^{m}\binom{a_r}{2}
+
E\sum_{1\le r<s\le m}a_ra_s,
\]
since every pair of elements of \(S\) either belongs to the same partition block or to two distinct partition blocks. Substituting these expressions into the constraints of
\eqref{eq:pairwise_dual} yields \eqref{eq:reduced_dual}. Since the reduced dual is obtained from the original dual by imposing the symmetry restrictions in \eqref{al:symmetry_dual}, the feasible region of the reduced dual is a subset of that of the original dual, implying $
f^{++}(\mb{x}^{\mathrm{id}})
\le
f^{++}_{\mathrm{RD}}(\mb{x}^{\mathrm{id}})$.
\end{proof}
We next scale the reduced dual variables $B,D,E$ by $n,n^2,n^2$ respectively and derive an asymptotic reduced dual as $n \to \infty$ with the number of partition blocks $m$ growing sublinearly with \(n\) \emph{i.e.}, $m(n)=o(n)\to\infty
$. 
\begin{lemma}[Scaled asymptotic reduced dual (SARD) formulation]
\label{lem:reduced_dual_scaled_asymptotic}
\noindent Under the asymptotic regime $
n\to\infty,\; m(n)=o(n)\to\infty
$, introduce the scaled dual variables
\begin{equation}\label{eq:scaled_variables}
\widetilde{B}=\frac{B}{n},
\qquad
\widetilde{D}=\frac{D}{n^2},
\qquad
\widetilde{E}=\frac{E}{n^2},
\end{equation}
and, for notational convenience, subsequently relabel
$\widetilde{B},\widetilde{D},\widetilde{E}$ as $B,D,E$. Then the reduced dual \eqref{eq:reduced_dual} is asymptotically equivalent to
\begin{equation}\label{eq:reduced_dual_scaled_asymptotic}
\begin{array}{lll}
f^{++}_{\footnotesize\mathrm{SARD}}(\mb{x}^{\mathrm{id}})
=\displaystyle
\underset{\footnotesize A,B,C,D}{\scalebox{1.1}{$\min$}}
&
\displaystyle
A
+\frac{B}{m}
+\frac{D}{2m^3}
+\frac{1}{2m^2}\left( 1-\frac1m \right)E
\\[4ex]
\mathrm{s.t.}
&
\displaystyle
A
+\frac{B}{m}\sum_{r=1}^{m}u_r
+\frac{D}{2m^2}\sum_{r=1}^{m}u_r^2
+\frac{E}{m^2}
\sum_{1\le r<s\le m}u_ru_s
\ge
1-\prod_{r=1}^{m}(1-u_r),&\forall\,\mb{u}\in[0,1]^m.
\\[2ex]
\end{array}
\end{equation}
\end{lemma}

\begin{proof}
The objective in \eqref{eq:reduced_dual} can be written in terms of the scaled variables in \ref{eq:scaled_variables} as:
\begin{align}
\label{al:objective_asymptotic_dual}
\Phi_n
=
A
+\frac{B}{n}\frac{n}{m}
+\frac{D}{n^{2}m^{2}}
\left((m-1)\binom{\ell}{2}+\binom{\ell_{m}}{2}\right)
+\frac{E}{n^{2}m^{2}}
\left(\binom{m-1}{2}\ell^{2}+(m-1)\ell\ell_{m}\right).
\end{align}
\noindent Since 
$
\ell=\left\lfloor\dfrac{n}{m}\right\rfloor
=\dfrac{n}{m}+O(1),
$
it follows that
$
\ell_{m}
=
n-(m-1)\ell
=
\dfrac{n}{m}+O(m),
$
and hence, under the condition
$ m(n)=o(n)$, we have :
\[
\frac1{n^{2}}\binom{\ell}{2}
=
\frac1{2m^{2}}+O\!\left(\frac1{mn}\right),
\qquad
\frac1{n^2}\binom{\ell_{m}}{2}
=
\frac1{2m^2}
+
O\!\left(\frac1n\right)
\]
\vskip 5pt
\[
\frac{\ell^2}{n^{2}}
=\frac{1}{m^{2}}
+
O\!\left(\frac{1}{mn}\right),\quad  
\frac{\ell\ell_{m}}{n^{2}}=
\frac1{m^2}
+
O\!\left(\frac1{n}\right)
\]
Substituting these  expansions into \eqref{al:objective_asymptotic_dual} under the asymptotic regime $
n\to\infty,\; m(n)=o(n)\to\infty
$ yields:
\[
\Phi_n
=A
+\frac{B}{m}
+\frac{D}{2m^{3}}
+\left( 1-\frac1m \right)\frac{E}{2m^2}
+O\!\left(\frac1{mn}\right)=
A
+\frac{B}{m}
+\frac{D}{2m^{3}}
+\left( 1-\frac1m \right)\frac{E}{2m^2}
+o(1).
\]

\noindent For the constraints, substituting the scaled variables into the reduced dual
constraint \eqref{eq:reduced_dual} gives
\begin{align}\label{al:constraint_asymptotic_dual}
A
+\frac{B}{n}\sum_{r=1}^{m}a_r
+\frac{D}{n^{2}}
\sum_{r=1}^{m}\binom{a_r}{2}
+\frac{E}{n^{2}}
\sum_{1\le r<s\le m}a_ra_s
\ge
\frac{f(S)}{\ell^{m-1}\ell_{m}}, \qquad \forall S \subseteq N.
\end{align}

\noindent 
Since $
|A_r|=\ell=\dfrac{n}{m}+O(1)$, for $r\in M \setminus \{m\}$ and 
$|A_m|=\ell_{m}=\dfrac{n}{m}+O(m)$, under the condition $
m(n)=o(n)
$, it follows that for any $r \in M$,
\[
\frac{|A_r|}{n}
=
\frac1m+O\!\left(\frac1n\right) \qquad \mbox{ and }\qquad  \frac{u_r|A_r|}{n}
=
\frac {u_r}m+O\!\left(\frac1n\right),\; \mbox{ since $u_r \in [0,1]$}
\]

\noindent Therefore, since $
a_r=|A_r|u_r$ and $u_r \in [0,1]$, we have 
\[
\frac{a_r}{n}
=
\frac{u_r}{m}
+
O\!\left(\frac1n\right),
\qquad
\frac1{n^2}\binom{a_r}{2}
=
\frac{|A_r|^2u_r^2}{2n^2}+O\left(\frac1{mn}\right)
\]
and \[
\frac1{n^2}a_ra_s
=
\frac1{n^2}|A_r||A_s|u_ru_s
=
\frac{1}{m^2}u_ru_s+O\left(\frac1n\right).
\]

\noindent Substituting these expansions into \eqref{al:constraint_asymptotic_dual}yields:
\[
A
+\frac{B}{m}\sum_{r=1}^{m}u_r
+\frac{D}{2m^{2}}
\sum_{r=1}^{m}u_r^{2}
+\frac{E}{m^{2}}
\sum_{1\le r<s\le m}u_ru_s
+O\left(\frac{m^2}{n}\right)\ge
\frac{f(S)}{\ell^{m-1}\ell_{m}}, \qquad \forall S \subseteq N.
\] and hence under the asymptotic regime 
$
n\to\infty,\; m(n)=o(n)\to\infty
$, we have
\begin{align}\label{al:constraint_discrete_u_r}
=A
+\frac{B}{m}\sum_{r=1}^{m}u_r
+\frac{D}{2m^{2}}
\sum_{r=1}^{m}u_r^{2}
+\frac{E}{m^{2}}
\sum_{1\le r<s\le m}u_ru_s
+o(1)
\ge
1-\prod_{r=1}^{m}(1-u_r)
\end{align}
where the inequality in \eqref{al:constraint_discrete_u_r} is required to hold for all  subsets $S\subseteq N$.
The family of occupancy vectors $
\mb{u}(S)=(u_1(S),\ldots,u_m(S))$ induced by all subsets \(S\subseteq N\) becomes \textit{dense} as \(n\to\infty\). Indeed, for any \(\mb{u}\in[0,1]^m\), choosing
\[
a_r=\lfloor u_r|A_r|\rfloor,\qquad r\in M,
\]
it is always possible to assemble a subset \(S\) by choosing any $a_r$ elements from each set $A_r$ and taking their union so that the normalized occupancy vector $\left(\dfrac{a_r}{|A_r|}:r \in M\right)$ approximates $\mb{u}$ arbitrarily closely. Hence, for all $r \in M$, we have
\[
0 \le u_r-\dfrac{a_r}{|A_r|}
\le
\dfrac1{|A_r|}, \qquad 
\]
where the difference converges to zero since \(|A_r|\to\infty\) for all $r \in M$ from \eqref{al:T_and_Q_size_infty}. 
Consequently, the set of constraints in \eqref{al:constraint_discrete_u_r} can be equivalently expressed as 
\begin{align}\label{al:constraint_continuous_u_r_hypercube}
=A
+\frac{B}{m}\sum_{r=1}^{m}u_r
+\frac{D}{2m^{2}}
\sum_{r=1}^{m}u_r^{2}
+\frac{E}{m^{2}}
\sum_{1\le r<s\le m}u_ru_s
\ge
1-\prod_{r=1}^{m}(1-u_r),\;\;\forall\,\mb{u}\in[0,1]^m.
\end{align}
For a fixed $S \subseteq N$, the RHS term in \eqref{al:constraint_continuous_u_r_hypercube} is the normalized coverage polynomial defined in \eqref{al:normalized_set_function} and can be viewed as the multilinear extension (defined earlier) of the submodular set function 
$f(R)=\min\{|R|,1\}$ with $R \subseteq M$ and marginal probabilities $u_r(S)$.
This completes the proof.
\end{proof}

\subsection{Proof of Theorem \ref{thm:main_result_asymptotic}}\label{sec:thm_main_result_asymptotic}
We are now ready to prove the main result of this section in Theorem \ref{thm:main_result_asymptotic} by utilizing an upper bound on $f^{++}(\mb{x})$ provided by the scaled asymptotic reduced dual (SARD) formulation in Lemma \ref{lem:reduced_dual_scaled_asymptotic}.
\begin{proof}
By Lemma's \ref{lem:reduced_dual} and \ref{lem:reduced_dual_scaled_asymptotic}, as
\(n\to\infty\), we have:
\[
f^{++}(\mb{x}^{\mathrm{id}})
\le
f^{++}_{\mathrm{SARD}}(\mb{x}^{\mathrm{id}}).
\]
From \eqref{eq:sufficient_ratio}, it is thus sufficient to construct a feasible solution for the scaled asymptotic reduced dual
in \eqref{eq:reduced_dual_scaled_asymptotic}
such that 
\begin{align*} 
f^{++}_{\footnotesize\mathrm{SARD}}(\mb{x}^{\mathrm{id}})
\longrightarrow
1-\frac1e.
\end{align*}

We express the (uncountably infinite) set of constraints in \eqref{eq:reduced_dual_scaled_asymptotic} as:
\begin{equation}
\label{eq:Q_F_feasibility}
\begin{gathered}
Q(\mb{u})\ge F(\mb{u}),
\qquad \forall\,\mb{u}\in[0,1]^m,
\\[5pt]
\mbox{where }\quad
Q(\mb{u})
=
A
+\frac{B}{m}\sum_{i=1}^{m}u_i
+\frac{D}{2m^2}\sum_{i=1}^{m}u_i^2
+\frac{E}{m^2}\sum_{1\le i<j\le m}u_i u_j,
\;\mbox{ and }\;
F(\mb{u})=1-\prod_{r=1}^{m}(1-u_r)
\end{gathered}
\end{equation}
are quadratic and $m^{\text{th}}$ degree polynomials in $\mb{u}$ respectively. We next construct an explicit feasible solution that satisfy these constraints by imposing the following conditions at the symmetric point of
the normalized occupancy vector $
\mb{u}^\star
=
\left(
\frac1m,\ldots,\frac1m
\right)$
 by:
 \begin{equation}
\begin{aligned}\label{al:second_order_conditions_Q_F}
Q(\mb{u}^\star)&=F(\mb{u}^\star),\\[3pt]
\nabla Q(\mb{u}^\star)&=\nabla F(\mb{u}^\star),\\[3pt]
\frac{\partial^2Q}{\partial u_i\partial u_j}(\mb{u}^\star)
&=
\frac{\partial^2F}{\partial u_i\partial u_j}(\mb{u}^\star),
\qquad
i\neq j.
\end{aligned}
\end{equation}
These conditions require $Q$ and $F$ to agree in their value, first
derivatives, and mixed second derivatives at $\mb{u}^\star$. Solving these equations yields expressions for the dual variables $A,B,E$ in terms of $D$ as follows:
\begin{equation}
\begin{array}{lll}\label{eq:asymptotic_dual_feasible_solution}
E
&=
-m^2\left(1-\dfrac1m\right)^{m-2},\\[10pt]
A
&=
\dfrac{D}{2m^3}
+
\left[1-
\left(1-\dfrac1m\right)^m
\right]-
\frac32
\left(1-\dfrac1m\right)^{m-1},\\[10pt]
B
&=
-
\dfrac{D}{m^2}
+
2m
\left(1-\dfrac1m\right)^{m-1}.
\end{array}
\end{equation}
It now remains to prove that the resulting quadratic polynomial $Q(\mb{u},D)$ satisfies \eqref{eq:Q_F_feasibility} for a finite free parameter \(D\), \emph{i.e.},
\[
\exists D \in \mathbb{R}:\;G(\mb{u},D)=Q(\mb{u},D)-F(\mb{u}) \ge 0, \;\;\forall \mb{u}\in[0,1]^m
\]
Since $A,B$ are affine in $D$ while $F$ is independent of $D$, we have: 
\begin{align*}
\frac{\partial G}{\partial D}
=
\frac{\partial Q}{\partial A}\frac{\partial A}{\partial D}+\frac{\partial Q}{\partial B}\frac{\partial B}{\partial D}+\frac{\partial Q}{\partial D}
&=
\frac1{2m^3}-\frac1{m^3}\sum_{i=1}^m u_i+\frac1{2m^2}\sum_{i=1}^m u_i^2
=
\frac1{2m^2}\sum_{i=1}^m\left(u_i-\frac1m\right)^2.
\end{align*}
Hence $G(\mb{u},D)$ is bi-affine in $\mb{u},D$ and can be expressed as:
\begin{align}\label{al:G_linear_LU}
G(\mb{u},D)=DL(\mb{u})+R(\mb{u}), \mbox{  where  }
L(\mb{u})=\frac1{2m^2}\sum_{i=1}^m\left(u_i-\frac1m\right)^2  
\end{align}

\noindent since $R$ is independent of $D$ and $D$ is independent of $\mb{u}$. The dual feasibility conditions in \eqref{eq:Q_F_feasibility} can now be expressed equivalently as
\begin{align}\label{al:D>=D_star}
G(\mb{u},D)=DL(\mb{u})+R(\mb{u})\ge0
\quad \forall\,\mb{u}\in[0,1]^m
\qquad\Longleftrightarrow\qquad
D\ge D^\star=\sup_{\mb{u}\in[0,1]^m}
\left(-\frac{R(\mb{u})}{L(\mb{u})}\right)
\end{align}
where the existence of $D^\star <\infty$ is sufficient to ensure feasibility since the dual variable $D$ is unrestricted in \eqref{eq:reduced_dual_scaled_asymptotic}. The conditions imposed in \eqref{al:second_order_conditions_Q_F} imply that $G(\mb{u},D)$ must satisfy:
\begin{align}\label{al:G_stationary_conditions}
G(\mb{u}^\star,D)=0,\qquad \nabla_u G(\mb{u}^\star,D)=0,\qquad
\frac{\partial^2G(\mb{u}^\star,D)}{\partial u_i\partial u_j}=0,
\qquad
i\neq j
\end{align}
From \eqref{al:G_linear_LU}, we have $
L(\mb{u}^\star)=0,\; \nabla L(\mb{u}^\star)=0,\; \dfrac{\partial^2L}{\partial u_i\partial u_j}(\mb{u}^\star)=0,
\;
i\neq j$, which along with \eqref{al:G_stationary_conditions} implies
\begin{align}
\label{al:R_conditions_1}
R(\mb{u}^\star)=0,\qquad \nabla R(\mb{u}^\star)=0, \qquad \dfrac{\partial^2R(\mb{u}^\star)}{\partial u_i\partial u_j}=0,
\;
i\neq j.
\end{align}
For the diagonal second derivatives, $
\dfrac{\partial^2L}{\partial u_i^2}=\dfrac1{m^2}
$ and hence for each $i \in M$, we have:
\begin{equation}\label{eq:G_pure_derivatives1}
    \begin{array}{lll}
\dfrac{\partial^2G(\mb{u},D)}{\partial u_i^2}=
D\dfrac{\partial^2L(\mb{u})}{\partial u_i^2}+\dfrac{\partial^2R(\mb{u})}{\partial u_i^2}
=\dfrac{D}{m^2}+\dfrac{\partial^2R(\mb{u})}{\partial u_i^2}.
\end{array}
\end{equation}
Further, from \eqref{eq:Q_F_feasibility}, we have $
\dfrac{\partial^2Q}{\partial u_i^2}=\dfrac{D}{m^2}$ and $\dfrac{\partial^2F}{\partial u_i^2}=0$ leading to
\begin{equation}
\label{eq:G_pure_derivatives2}
    \begin{array}{lll}
\dfrac{\partial^2G(\mb{u},D)}{\partial u_i^2}=
\dfrac{\partial^2Q}{\partial u_i^2}-\dfrac{\partial^2F}{\partial u_i^2}
=\dfrac{D}{m^2}.
\end{array}
\end{equation}
From \eqref{eq:G_pure_derivatives1} and \eqref{eq:G_pure_derivatives2}, we have $\dfrac{\partial^2R(\mb{u})}{\partial u_i^2}=0$, which together with \eqref{al:R_conditions_1} implies:
\begin{align}
\label{al:R_conditions_2}
R(\mb{u}^\star)=0,\qquad \nabla R(\mb{u}^\star)=0,\qquad \nabla^2R(\mb{u}^\star)=0.
\end{align}
Since $L(\mb{u})$ and $R(\mb{u})$ both vanish at $\mb{u}=\mb{u}^\star$ and $L(\mb{u})>0$ otherwise, we need to show that the ratio $-R(\mb{u})/L(\mb{u})$ in \eqref{al:D>=D_star}, although not defined at  $\mb{u}=\mb{u}^\star$, remains bounded both near and away from $\mb{u}^\star$. In particular, it suffices to establish the following two properties of the absolute ratio:
\begin{enumerate}
\item
\[
\lim_{\mb{u}\to\mb{u}^\star}
\left|
\frac{R(\mb{u})}{L(\mb{u})}
\right|<\infty,
\]
\item
\[
\sup_{\mb{u}\in K_{\epsilon}}
\left|
\frac{R(\mb{u})}{L(\mb{u})}
\right|<\infty,
\qquad \mbox{where }
\;K_\epsilon:=\left\{\mb{u}\in[0,1]^m:\|\mb{u}-\mb{u}^\star\|\ge\epsilon\right\},
\; \mbox{for any }\epsilon>0.
\]
\end{enumerate}
The first property follows from \eqref{al:R_conditions_2}
since it implies that the Taylor's expansion of $R$ about $\mb{u}=\mb{u}^\star$
satisfies
\[
R(\mb{u}^\star+\mb{v})=O(\|\mb{v}\|^3).
\] where $\mb{v}=\mb{u}-\mb{u}^\star$ denotes the perturbation around $\mb{u}^\star$. From \eqref{al:G_linear_LU}, we have $
L(\mb{u}^\star+\mb{v})=\Theta(\|\mb{v}\|^2)
$
and thus
\[
\dfrac{R(\mb{u}^\star+\mb{v})}{L(\mb{u}^\star+\mb{v})}
=O(\|\mb{v}\|)
\longrightarrow 0
\qquad\text{as }\mb{v}\to\mb{0}.
\]
To establish the second property, we note that
\[
L(\mb{u})
=
\frac{1}{2m^2}\|\mb{u}-\mb{u}^\star\|^2
\ge
\frac{\epsilon^2}{2m^2},
\qquad \forall\,\mb{u}\in K_\epsilon.
\]
Moreover, $R(\mb{u})=G(\mb{u},D)-D L(\mb{u})$ is a polynomial in $\mb{u}$ and hence is continuous on $[0,1]^m$. By the Extreme Value Theorem, there exists a finite constant $M$ such that
$
|R(\mb{u})|\le M,\;\forall\,\mb{u}\in[0,1]^m.
$
Therefore, for all $\mb{u}\in K_\epsilon$,
\[
\left|
\frac{R(\mb{u})}{L(\mb{u})}
\right|
\le
\frac{2m^2M}{\epsilon^2}<\infty.
\]
Combining the two properties, we obtain
\begin{align}\label{al:D_star_finite}
D^\star=\sup_{\mb{u}\in[0,1]^m\setminus\{\mb{u}^\star\}}
\left(-\dfrac{R(\mb{u})}{L(\mb{u})}\right)<\infty.
\end{align}

Substituting the expressions for $A,B,E$ from \eqref{eq:asymptotic_dual_feasible_solution} into the objective yields:
\begin{equation}\label{eq:f_++=zenclusen}
    \begin{array}{ll}
f^{++}_{\mathrm{SARD}}(\mb{x}^{\mathrm{id}})
& \le A
+\dfrac{B}{m}
+\dfrac{D}{2m^3}
+\dfrac{1}{2m^2}\left( 1-\dfrac1m \right)E\\[16pt]
&=1-
\left(
1-\dfrac1m
\right)^m
    \end{array}
\end{equation}
where the free parameter \(D\) cancels identically. For the $m$-partition instance with identity marginals $\mb{x}^{\mathrm{id}}$, the feasible solution constructed above gives
\[
\frac{f^+(\mb{x}^{\mathrm{id}})}{f^{++}(\mb{x}^{\mathrm{id}})}
\ge
\frac{1}{1-\left(1-\frac1m\right)^m}
\longrightarrow
\frac{e}{e-1}
\qquad\text{as }m\to\infty.
\]
Since $f^+(\mb{x}^{\mathrm{id}})=1$, it consequently follows that
$
f^{++}(\mb{x}^{\mathrm{id}})
\longrightarrow
1-\frac1e
$
in the asymptotic regime $
n\to\infty,\;m(n)\to\infty,\; m(n)=o(n)
$ and the proof is completed.
\end{proof}
\subsection{Structure of the Asymptotic worst case Pairwise Independent Distribution}\label{sec:boros_prekopa_asymptotic}
The asymptotic dual construction in Section \eqref{subsec:pairwise_scaled_asymptotic_dual} was sufficient to establish
Theorem~\ref{thm:main_result_asymptotic}. However, it does not reveal the structure of the
pairwise independent distribution attaining the asymptotic worst case. The following result, based on the construction of Boros and Pr\'ekopa in~\cite{boros_prekopa89}, shows that the extremal distribution that attains the limit
\[
f^{++}(\mb{x}^{\mathrm{id}})
\longrightarrow
1-\frac{1}{e}
\]
is asymptotically supported on subsets of two adjacent cardinalities, both of the form $n/m+O(1)$. Moreover, this distribution induces an expected occupancy of $1/m$ in each partition block $A_r,\;r \in M$, matching the identical marginal probabilities $x_i=1/m$.
\begin{lemma}[Asymptotic worst case distribution for $f^{++}(\mb{x}^{\mathrm{id}})$]\label{lem:pairwise_primal_bp_two_support}

\[
\textrm{Let }\; x=\frac1m,
\qquad
j=\left\lceil (n-1)x\right\rceil
=\left\lceil \frac{n-1}{m}\right\rceil.
\]
Then, under the asymptotic regime 
$
n\to\infty,\; m(n)=o(n)\to\infty
$, there exists a pairwise independent distribution \(\mb{p}_n^\star\) supported on subsets of size \(j-1\) and \(j\), \emph{i.e.},
\[
\mb{p}_n^\star(S)=
\begin{cases}
\alpha_n, & |S|=j-1,\\
1-\alpha_n, & |S|=j,\\
0, & \text{otherwise},
\end{cases}
\]
\vskip 10 pt
where \(\alpha_n=\dfrac{\left(1-\frac1m\right)\left(j-\frac{(n-1)}{m}\right)}{\binom{n-1}{j-1}}\in[0,1]\), such that
\[
f^{++}(\mb{x}^{\mathrm{id}})=\mathbb E_{\mb{p}_n^\star}[f(S)/|\mathcal H_m|]
=
1-\left(1-\frac1m\right)^m+o(1).
\]
\end{lemma}

\begin{proof}[Proof of Lemma~\ref{lem:pairwise_primal_bp_two_support}]\label{proof:lem:pairwise_primal_bp_two_support}
The worst case distribution $\mb{p}_n^\star$ is derived from an asymptotic reduction of the optimal distribution of an aggregated linear program with given first and second moments in \cite{boros_prekopa89}. This distribution was shown to be feasible for pairwise independent identical random elements in \cite{ramachandra_natarajan_SIDMA} (see Theorem $4.1$ therein). In other words,
\begin{equation}
\label{eq:pn_star}
\mb{p}_n^\star(S)=
\begin{cases}
\dfrac{(1-x)\bigl(j-(n-1)x\bigr)}{\binom{n-1}{j-1}},
& |S|=j-1,\\[2.5ex]
\dfrac{(1-x)\bigl(1+(n-1)x-j\bigr)}{\binom{n-1}{j}},
& |S|=j,\\[2.5ex]
\dfrac{n(n-1)x^2+(j-1)(j-2nx)}{(n-j)^2+(n-j)},
& |S|=n,\\[2.5ex]
0, & \text{otherwise}.
\end{cases}
\end{equation}
satisfies pairwise independence with identical marginals $x=1/m$ when
$j=\left\lceil (n-1)/m\right\rceil$.

For a fixed feature $H\in\mathcal H_m$ with $|H|=m$, the distribution in \eqref{eq:pn_star} is symmetric within each cardinality layer. Hence, conditional on $|S|=k$, it is uniform over all $k$-subsets of $N$, and
\begin{align}
\label{al:null_intersection}
\Pr\nolimits_{\mb{p}_n^\star}(H\cap S=\emptyset\mid |S|=k)
=
\frac{\binom{n-m}{k}}{\binom{n}{k}}
=
\prod_{r=0}^{m-1}\left(1-\frac{k}{n-r}\right).
\end{align}
Since
\[
\frac{n-1}{m}\le j<\frac{n-1}{m}+1,
\]
we have
\[
j=\frac{n}{m}+O(1),\qquad
\frac{j}{n}=\frac1m+O\!\left(\frac1n\right),\qquad
\frac{j-1}{n}=\frac1m+O\!\left(\frac1n\right).
\]
Thus, for $k\in\{j-1,j\}$ and uniformly for $0\le r\le m-1$,
\[
\frac{k}{n-r}
=
\frac1m+O\!\left(\frac1n\right).
\]
Therefore,
\[
\begin{aligned}
\log\left(
\prod_{r=0}^{m-1}\left(1-\frac{k}{n-r}\right)
\right)
&=
\sum_{r=0}^{m-1}
\log\left(1-\frac1m+O\!\left(\frac1n\right)\right)\\
&=
m\log\left(1-\frac1m\right)
+O\!\left(\frac{m}{n}\right)\\
&=
m\log\left(1-\frac1m\right)+o(1),
\end{aligned}
\]
where the second equality follows from a first-order Taylor expansion of $\log$ around $1-\frac1m$ and the third equality follows from the asymptotic regime $
n\to\infty,\; m(n)=o(n)
$.
Exponentiating,
\[
\begin{aligned}
\prod_{r=0}^{m-1}\left(1-\frac{k}{n-r}\right)
&=
\left(1-\frac1m\right)^m e^{o(1)}
=
\left(1-\frac1m\right)^m\bigl(1+o(1)\bigr)\\
&=
\left(1-\frac1m\right)^m+o(1),
\end{aligned}
\]
since $\left(1-\frac1m\right)^m \to 1/e$. Hence,
\[
\Pr\nolimits_{\mb{p}_n^\star}(H\cap S=\emptyset\mid |S|=k)
=
\left(1-\frac1m\right)^m+o(1),
\qquad k\in\{j-1,j\}.
\]

We next show that the mass on $|S|=n$ vanishes. Using $x=1/m$ and
$j=(n-1)x+O(1)$, we can write $j=(n-1)x+\delta$ with $0\le\delta<1$. Then
\[
\begin{aligned}
n(n-1)x^2+(j-1)(j-2nx)
&=
n(n-1)x^2+\bigl((n-1)x+\delta-1\bigr)\bigl(\delta-(n+1)x\bigr)\\
&=(n-1)x^2+(n+1-2\delta)x+\delta(\delta-1)
=O(n).
\end{aligned}
\]
Thus, the quadratic terms cancel. Moreover, since $m\ge2$,
$j=\left\lceil(n-1)/m\right\rceil\le n/2+1$, so $n-j=\Theta(n)$ and
\[
(n-j)^2+(n-j)=\Theta(n^2).
\]
Therefore,
\[
\Pr\nolimits_{\mb{p}_n^\star}(|S|=n)
=
\frac{O(n)}{\Theta(n^2)}
=
O\!\left(\frac1n\right)
=o(1).
\]

Consequently, asymptotically all probability mass is concentrated on subsets
of sizes $j-1$ and $j$. Hence, for any fixed feature $H$,
\[
\Pr\nolimits_{\mb{p}_n^\star}(H\cap S=\emptyset)
=
\left(1-\frac1m\right)^m+o(1),
\]
and therefore
\[
\Pr\nolimits_{\mb{p}_n^\star}(H\cap S\neq\emptyset)
=
1-\left(1-\frac1m\right)^m+o(1).
\]
Since every feature in $\mathcal H_m$ is statistically identical, the expected
contribution of each feature to the objective is
\[
1-\left(1-\frac1m\right)^m+o(1).
\]
By linearity of expectation, summing over
$|\mathcal H_m|=\ell^{m-1}\ell_{m}$ features yields
\[
\mathbb E_{\mb{p}_n^\star}[f(S)]
=
\ell^{m-1}\ell_{m}\left(1-\left(1-\frac1m\right)^m\right)
+o(\ell^{m-1}\ell_{m}),
\]
and hence
$
\mathbb E_{\mb{p}_n^\star}[f(S)/|\mathcal H_m|]
=
1-\left(1-\frac1m\right)^m+o(1),
$
as claimed.
\end{proof}
 We next show that the asymptotically extremal pairwise independent distribution in Lemma \ref{lem:pairwise_primal_bp_two_support} induces an expected occupancy of $1/m$ in each partition block, matching the identity marginals $x_i=1/m$.
\begin{corollary}[Asymptotic occupancy of the partition blocks]\label{cor:asymptotic_occupancy_partition_blocks}
Under the asymptotic regime
\[
n\to\infty,\qquad m=m(n)\to\infty,\qquad m(n)=o(n),
\]
the distribution $\mb{p}_n^\star$ in \eqref{eq:pn_star} induces an asymptotic occupancy
of $1/m$ in each partition block $A_r$, $r\in M$, matching the identical marginal probabilities $x_i=1/m$, \emph{i.e.},
\[
\mathbb{E}_{\mb{p}_n^\star}\!\left[u_r(S)\mid |S|=k\right]
\longrightarrow \frac1m,
\qquad k\in\{j-1,j\},
\]
where $u_r(S)=\dfrac{|S\cap A_r|}{|A_r|}$ denotes the normalized occupancy of block $A_r$ defined earlier.
\end{corollary}
\begin{proof}
    The proof is relegated to Appendix~\ref{proof_cor:asymptotic_occupancy_partition_blocks}. 
\end{proof}

\section{Conclusion}\label{sec:conclusion}

This paper resolves two open questions concerning the worst case pairwise
independent correlation gap for monotone submodular functions.
First, we establish that the $4/3$ bound indeed holds for $n=4$, thereby
confirming the conjecture of \cite{ramachandra_special_cases_OR_letters}
for the only dimension left unresolved after the counterexample for
$n\geq5$ in \cite{ramachandra_natarajan_counter_2026}. Using an AI-assisted proof combining theoretical analysis and computational verification, we develop a primal based approach that combines a structural characterization of
optimal numerator vertices with permutation symmetry, cone certificate
systems, Bernstein polynomial representations, recursive simplex
subdivision, and computational verification. By treating both the
numerator and denominator through their primal
formulations, our approach allows a common
cone certificate system to be constructed. In particular, $2,745$ Bernstein coefficient systems were computationally verified to be  feasible, yielding $4,476$ successful leaf certificates
and establishing the $4/3$ bound universally for $n=4$, and the bound is
tight.

Second, we complement the counterexample in
\cite{ramachandra_natarajan_counter_2026} by showing that although the
$4/3$ bound fails for $n\geq5$, the worst case pairwise independent
correlation gap fails to improve upon the classical
$e/(e-1)$ bound established in \cite{Agrawal2012} under mutual independence.
We do so by partitioning the ground set into $m$ blocks and defining
features that select one element from each block, together with a
monotone submodular union coverage function that counts the features
covered by a given set $S\subseteq N$. The key condition is that the number of
blocks grows sublinearly with the ground set size, so that both the number
of blocks and the size of each block tend to infinity. The proof technique uses a scaled asymptotic reduced dual (SARD) of the
pairwise independent extension and constructs a feasible solution that
ensures that the pairwise independent extension is asymptotically no
larger than $1-1/e$. 
On the primal side, we characterize the corresponding asymptotically extremal pairwise independent distribution, showing that it is supported on subsets of two adjacent
cardinalities, both of order $n/m$, and induces an expected occupancy of
$1/m$ in each partition block, matching the identity marginals
$x_i=1/m$. In particular, the key insight from this result is that, despite constraining
only pairwise marginals and allowing arbitrary higher order dependencies,
pairwise independence can be \emph{as restrictive} as mutual independence in the
worst case. 
The asymptotic result immediately extends to $t$-wise independent random elements
($t\ge2$), since $t$-wise independence implies pairwise independence. 

These two results together add to the theory of the correlation gap by
showing that the pairwise independent correlation gap exhibits a dimension dependent worst case
bound, with $n=4$ being the largest dimension for which the $4/3$ bound
holds, while the worst case gap asymptotically attains the $e/(e-1)$ bound. For dimensions $5\le n<\infty$, the precise worst
case pairwise independent correlation gap remains unknown and is bounded
below by the current best known lower bound $640/479>4/3$ and above by $e/(e-1)$. Another natural direction for future work is to consider positive correlations,
which, being less restrictive, allow a larger set of feasible distributions
thus increasing the denominator, suggesting that the $4/3$ bound may
continue to hold under positive correlations. 
\paragraph{Acknowledgment:} The author would like to thank Gustav Malmqvist\footnote{Gustav Malmqvist, AI enthusiast. LinkedIn: \url{https://se.linkedin.com/in/gmalmqvist}.}, an AI enthusiast, who, inspired by  our explicit identification of the $n=4$ case as an open problem in \cite{ramachandra_natarajan_counter_2026}, attempted to solve it using a human-AI workflow. Using GPT 5.6 and Fable 5, he developed a high level proof idea that ultimately led to the approach pursued in this paper.

\newpage
\begin{appendices}
\section{Proofs of Results}\label{appendix:proof_of_results}

\begin{appproof}[Proof of Lemma~\ref{lem:vertex_support_characterization}.]\label{proof:lem_vertex_support_characterization}

First suppose that $\mb{u}$ is a vertex with $J=\operatorname{supp}(\mb{u})$. Suppose, to the contrary that the columns of $A$ (in $\mathbb{R}^5$) corresponding to the subsets in $J$ are 
linearly dependent. Then, there exists a vector $\mb{d} \in \mathbb{R}^{16}$ with $\;\mb{d} \neq \mb{0}$ and $\operatorname{supp}(\mb{d})=J$, such that
$
A\mb{d}=\mb{0}.
$
Since $u_j>0$ for every $j\in J$, choose
\[
0<\epsilon<
\min_{j\in J:\,d_j\neq0}\frac{u_j}{|d_j|}.
\]
Then $
\mb{u}+\epsilon\mb{d}\ge0,\;
\mb{u}-\epsilon\mb{d}\ge0,
$
and
\[
A(\mb{u}\pm\epsilon\mb{d})
=
A\mb{u}
=
\begin{pmatrix}
1\\
\mb{x}
\end{pmatrix}.
\]
Thus both perturbed points belong to $\mathcal{F}_{\mathrm{U}}(\mb{x})$,
with
\[
\mb{u}
=
\frac12(\mb{u}+\epsilon\mb{d})
+
\frac12(\mb{u}-\epsilon\mb{d}),
\]
contradicting the extremality of $\mb{u}$. Hence the corresponding columns of $A$ (in $\mathbb{R}^5$) are 
linearly independent and hence
$|J|\le5$. Thus every vertex of $\mb{u} \in \mathcal{F}_{\mathrm U}(\mb{x})$ is supported on at most
five subsets of $N_4$.

Conversely, suppose that the columns of $A$ corresponding to $J$ are
linearly independent. Then $|J| \le 5$. If $\mb{u}$ were not a vertex, there would exist $\mb{u}^+,\mb{u}^-\in\mathcal{F}_{\mathrm{U}}(\mb{x})$ such that
\[
\mb{u}
=
\frac12(\mb{u}^++\mb{u}^-),\quad \mb{u}^+ \neq \mb{u}^-
\]
Since $\mb{u}$ is zero outside $J$ and $\mb{u}^+,\mb{u}^-\ge0$, both
$\mb{u}^+$ and $\mb{u}^-$ must also be zero outside $J$. Hence there exists a vector $\mb{d}$ such that:
\[
\begin{aligned}
\mb{d}
&=
\mb{u}^+-\mb{u}^-
\neq \mb{0},
\quad
\operatorname{supp}(\mb{d}) \subseteq J,\\
\mbox{and }\quad A\mb{d}
&=
A\mb{u}^+-A\mb{u}^-
=
\mb{0}.
\end{aligned}\]
which contradicts the linear independence of the columns indexed by
$J$. Therefore $\mb{u}$ is a vertex. 
\end{appproof}

\begin{appproof}[Proof of Corollary \ref{cor:nonsingular_simplices_cover}\label{proof:cor_nonsingular_simplices_cover}]
The covering property follows immediately from
Corollary~\ref{cor:common_convex_coefficients}, since every
$\mb{x}\in[0,1]^4$ belongs to the incidence simplex corresponding to
the nonsingular five-subset basis induced by an optimal numerator
vertex. To see that the covering is overlapping, consider the two
nonsingular bases
$$B_1=\{\emptyset,\{4\},\{3,4\},\{2,3,4\},\{1,2,3,4\}\},\quad B_2=\{\emptyset,\{1\},\{2\},\{3\},\{4\}\}$$. whose corresponding incidence
simplices are
\[
\Delta_V(B_1)
=
\operatorname{conv}\{\mb{0},\mb{e}_4,\mb{e}_3+\mb{e}_4,
\mb{e}_2+\mb{e}_3+\mb{e}_4,\mb{1}\},
\qquad
\Delta_V(B_2)
=
\operatorname{conv}\{\mb{0},\mb{e}_1,\mb{e}_2,\mb{e}_3,\mb{e}_4\}.
\]
The first simplex is characterized by
$x_1\leq x_2\leq x_3\leq x_4$, while the second is characterized by
$x_1+x_2+x_3+x_4\leq1$. Hence their intersection contains all
$\mb{x}$ satisfying both conditions and is therefore non-empty containing a continuum of points. Thus, the set of nonsingular five-subset simplices cover
$[0,1]^4$ but do not form a partition.
\end{appproof}

\begin{appproof}
[Proof of Lemma \ref{lem:symmetry_permutation_invariance}\label{proof:lem_symmetry_permutation_invariance}]
A permutation $\pi$ of the four ground set elements induces permutations
of the $16$ subsets of $N_4$, the $56$ submodularity constraints, and the
$11$ moment equations. Let
\[
R_\pi\in\mathbb{R}^{16\times16},\qquad
L_\pi\in\mathbb{R}^{56\times56},\qquad
T_\pi\in\mathbb{R}^{11\times11}
\]
denote the corresponding permutation matrices that act on the Bernstein coefficient vectors
according to
\begin{equation}\label{eq:permuted_matrices_acting_Bernstein}
R_\pi\mb{p}^{ab}=\mb{p}^{ab}_{\pi},
\qquad
R_\pi\mb{u}^{{\star}^{ab}}
=\mb{u}^{{\star}^{ab}}_{\pi},
\qquad
R_\pi\mb{e}_{\emptyset}=\mb{e}_{\emptyset},\qquad
L_\pi\mb{\gamma}^{ab}
=\mb{\gamma}^{ab}_{\pi},
\qquad
T_\pi\mb{r}^{ab}
=\mb{r}^{ab}_{\pi}.
\end{equation}
where, $R_\pi\mb{e}_{\emptyset}=\mb{e}_{\emptyset}$ follows from the first row and first column of $R_\pi$ being $(1,0,\ldots,0)$, since the empty set is fixed under every permutation. Thus $R_\pi$ acts on the
$16$-dimensional distribution vectors $\mb{u}^{{\star}^{ab}}$ and
$\mb{p}^{ab}$, $L_\pi$ acts on the $56$-dimensional vector
$\mb{\gamma}^{ab}$, and $T_\pi$ acts on the $11$-dimensional
right-hand side $\mb{r}^{ab}$ in
\eqref{eq:coeff_wise_27_equations_73_variables}. Since relabeling the
ground set elements only relabels the marginal and pairwise moment
equations, the moment matrix satisfies
\begin{equation}\label{eq:permuted_matrices_A_P}
A_{\mathrm P}R_\pi=T_\pi A_{\mathrm P}.
\end{equation}
Likewise, relabeling the ground set elements only relabels the
submodularity constraints. Hence the rows and columns of $G^\top$ are
permuted consistently, giving
\begin{equation}\label{eq:permuted_matrices_G}
R_\pi G^\top=G^\top L_\pi.
\end{equation}
 From \eqref{eq:permuted_matrices_acting_Bernstein}, \eqref{eq:permuted_matrices_A_P}, \eqref{eq:permuted_matrices_G}, the moment equation for the permuted simplex is satisfied since
\[
A_{\mathrm P}\mb{p}^{ab}_{\pi}
=
A_{\mathrm P}(R_\pi\mb{p}^{ab})
=
T_\pi A_{\mathrm P}\mb{p}^{ab}
=
T_\pi\mb{r}^{ab}=
\mb{r}^{ab}_{\pi}
.\]
and the cone certificate equation for the permuted simplex is
satisfied since
\[
\begin{array}{rcl}
\displaystyle
\frac{4}{3}\mb{p}^{ab}_{\pi}
-
G^\top\mb{\gamma}^{ab}_{\pi}
-
c^{ab}\mb{e}_{\emptyset}
&=&
\displaystyle
\frac{4}{3}R_\pi\mb{p}^{ab}
-
G^\top L_\pi\mb{\gamma}^{ab}
-
c^{ab}\mb{e}_{\emptyset}
\\[1ex]
&=&
\displaystyle
\frac{4}{3}R_\pi\mb{p}^{ab}
-
R_\pi G^\top\mb{\gamma}^{ab}
-
c^{ab}R_\pi\mb{e}_{\emptyset}
\\[1ex]
&=&
\displaystyle
R_\pi\left(
\frac{4}{3}\mb{p}^{ab}
-
G^\top\mb{\gamma}^{ab}
-
c^{ab}\mb{e}_{\emptyset}
\right)
\\[1ex]
&=&
R_\pi\mb{u}^{{\star}^{ab}}
=
\mb{u}^{{\star}^{ab}}_{\pi}.
\end{array}
\]
Thus, permutation of a feasible coefficient solution for the
representative simplex yields a feasible coefficient solution for every
simplex in the same orbit.
\end{appproof}

\begin{appproof}[Proof of Corollary~\ref{cor:asymptotic_occupancy_partition_blocks}]\label{proof_cor:asymptotic_occupancy_partition_blocks}
For $k\in\{j-1,j\}$, the distribution in \eqref{eq:pn_star} is uniform within the
$k$-layer. Hence, for any $i\in N$,
\[
\Pr\nolimits_{\mb{p}_n^\star}(i\in S\mid |S|=k)
=
\frac{\binom{n-1}{k-1}}{\binom{n}{k}}
=
\frac{k}{n}
=
\frac1m+O\!\left(\frac1n\right).
\]
For each partition block $A_r$, $r\in M$, we have
$|S\cap A_r|=\sum_{i\in A_r}\mathbbm{1}_{\{i\in S\}}$, and hence, by linearity of expectation,
\[
\begin{aligned}
\mathbb{E}_{\mb{p}_n^\star}\!\left[u_r(S)\mid |S|=k\right]
&=
\frac{1}{|A_r|}\sum_{i\in A_r}
\Pr\nolimits_{\mb{p}_n^\star}(i\in S\mid |S|=k)
=
\frac{k}{n}
\longrightarrow\frac1m
\end{aligned} 
\] under the asymptotic regime considered. Thus, the expected occupancy of each partition block $A_r$, $r\in M$, which is the average probability that an element of $A_r$ is included in $S$, equals the identical marginal probability $\Pr(i\in S)=1/m$.
\end{appproof}

\section{Orbit representatives and subdivision statistics from Section \ref{subsec:permutation_orbits}}
\label{appendix_sec:orbit_representatives_and_subdivision_stats}
\newgeometry{top=0.6cm,bottom=0.8cm,left=1cm,right=1cm}
\begin{table}[htbp]
\centering
\caption{Subdivision statistics for the $183$ permutation orbit representatives.}
\label{tab:orbit_subdivision_statistics}
\scriptsize
\setlength{\tabcolsep}{3pt}
\renewcommand{\arraystretch}{0.7}
\begin{tabular}{c p{0.30\linewidth} c c @{\hspace{0.35cm}} c p{0.30\linewidth} c c}
\hline
Orbit & Representative basis & Leaves & Depth & Orbit & Representative basis & Leaves & Depth \\
\hline
1 & $\{\emptyset, \{1\}, \{2\}, \{3\}, \{4\}\}$ & 32 & 7 & 93 & $\{\{1\}, \{2\}, \{1,3\}, \{1,2,3\}, \{2,3,4\}\}$ & 15 & 6 \\
2 & $\{\emptyset, \{1\}, \{2\}, \{3\}, \{1,4\}\}$ & 33 & 6 & 94 & $\{\{1\}, \{2\}, \{1,3\}, \{1,2,3\}, \{1,2,3,4\}\}$ & 13 & 5 \\
3 & $\{\emptyset, \{1\}, \{2\}, \{3\}, \{1,2,4\}\}$ & 10 & 4 & 95 & $\{\{1\}, \{2\}, \{1,3\}, \{1,4\}, \{3,4\}\}$ & 37 & 8 \\
4 & $\{\emptyset, \{1\}, \{2\}, \{3\}, \{1,2,3,4\}\}$ & 24 & 6 & 96 & $\{\{1\}, \{2\}, \{1,3\}, \{1,4\}, \{1,2,3,4\}\}$ & 25 & 6 \\
5 & $\{\emptyset, \{1\}, \{2\}, \{1,3\}, \{1,4\}\}$ & 29 & 7 & 97 & $\{\{1\}, \{2\}, \{1,3\}, \{2,4\}, \{3,4\}\}$ & 36 & 8 \\
6 & $\{\emptyset, \{1\}, \{2\}, \{1,3\}, \{2,4\}\}$ & 6 & 3 & 98 & $\{\{1\}, \{2\}, \{1,3\}, \{2,4\}, \{1,2,3,4\}\}$ & 22 & 5 \\
7 & $\{\emptyset, \{1\}, \{2\}, \{1,3\}, \{1,2,4\}\}$ & 16 & 5 & 99 & $\{\{1\}, \{2\}, \{1,3\}, \{1,2,4\}, \{3,4\}\}$ & 54 & 8 \\
8 & $\{\emptyset, \{1\}, \{2\}, \{1,3\}, \{3,4\}\}$ & 42 & 8 & 100 & $\{\{1\}, \{2\}, \{1,3\}, \{1,2,4\}, \{1,3,4\}\}$ & 17 & 5 \\
9 & $\{\emptyset, \{1\}, \{2\}, \{1,3\}, \{1,3,4\}\}$ & 5 & 3 & 101 & $\{\{1\}, \{2\}, \{1,3\}, \{1,2,4\}, \{2,3,4\}\}$ & 28 & 6 \\
10 & $\{\emptyset, \{1\}, \{2\}, \{1,3\}, \{2,3,4\}\}$ & 14 & 5 & 102 & $\{\{1\}, \{2\}, \{1,3\}, \{3,4\}, \{1,3,4\}\}$ & 21 & 8 \\
11 & $\{\emptyset, \{1\}, \{2\}, \{1,3\}, \{1,2,3,4\}\}$ & 14 & 5 & 103 & $\{\{1\}, \{2\}, \{1,3\}, \{3,4\}, \{2,3,4\}\}$ & 30 & 8 \\
12 & $\{\emptyset, \{1\}, \{2\}, \{1,2,3\}, \{1,2,4\}\}$ & 27 & 7 & 104 & $\{\{1\}, \{2\}, \{1,3\}, \{3,4\}, \{1,2,3,4\}\}$ & 41 & 7 \\
13 & $\{\emptyset, \{1\}, \{2\}, \{1,2,3\}, \{3,4\}\}$ & 50 & 7 & 105 & $\{\{1\}, \{2\}, \{1,3\}, \{1,3,4\}, \{1,2,3,4\}\}$ & 13 & 5 \\
14 & $\{\emptyset, \{1\}, \{2\}, \{1,2,3\}, \{1,3,4\}\}$ & 16 & 6 & 106 & $\{\{1\}, \{2\}, \{1,3\}, \{2,3,4\}, \{1,2,3,4\}\}$ & 23 & 5 \\
15 & $\{\emptyset, \{1\}, \{2\}, \{1,2,3\}, \{1,2,3,4\}\}$ & 12 & 5 & 107 & $\{\{1\}, \{2\}, \{1,2,3\}, \{1,2,4\}, \{3,4\}\}$ & 75 & 8 \\
16 & $\{\emptyset, \{1\}, \{1,2\}, \{1,3\}, \{1,4\}\}$ & 32 & 7 & 108 & $\{\{1\}, \{2\}, \{1,2,3\}, \{1,2,4\}, \{1,3,4\}\}$ & 31 & 6 \\
17 & $\{\emptyset, \{1\}, \{1,2\}, \{1,3\}, \{2,4\}\}$ & 13 & 5 & 109 & $\{\{1\}, \{2\}, \{1,2,3\}, \{1,2,4\}, \{1,2,3,4\}\}$ & 25 & 6 \\
18 & $\{\emptyset, \{1\}, \{1,2\}, \{1,3\}, \{1,2,4\}\}$ & 12 & 5 & 110 & $\{\{1\}, \{2\}, \{1,2,3\}, \{3,4\}, \{1,3,4\}\}$ & 39 & 7 \\
19 & $\{\emptyset, \{1\}, \{1,2\}, \{1,3\}, \{2,3,4\}\}$ & 28 & 6 & 111 & $\{\{1\}, \{2\}, \{1,2,3\}, \{3,4\}, \{1,2,3,4\}\}$ & 50 & 7 \\
20 & $\{\emptyset, \{1\}, \{1,2\}, \{1,3\}, \{1,2,3,4\}\}$ & 10 & 4 & 112 & $\{\{1\}, \{2\}, \{1,2,3\}, \{1,3,4\}, \{1,2,3,4\}\}$ & 16 & 5 \\
21 & $\{\emptyset, \{1\}, \{1,2\}, \{2,3\}, \{2,4\}\}$ & 55 & 9 & 113 & $\{\{1\}, \{1,2\}, \{1,3\}, \{2,3\}, \{1,4\}\}$ & 12 & 4 \\
22 & $\{\emptyset, \{1\}, \{1,2\}, \{2,3\}, \{1,2,4\}\}$ & 17 & 6 & 114 & $\{\{1\}, \{1,2\}, \{1,3\}, \{2,3\}, \{2,4\}\}$ & 18 & 6 \\
23 & $\{\emptyset, \{1\}, \{1,2\}, \{2,3\}, \{3,4\}\}$ & 31 & 8 & 115 & $\{\{1\}, \{1,2\}, \{1,3\}, \{2,3\}, \{1,2,4\}\}$ & 16 & 5 \\
24 & $\{\emptyset, \{1\}, \{1,2\}, \{2,3\}, \{1,3,4\}\}$ & 31 & 6 & 116 & $\{\{1\}, \{1,2\}, \{1,3\}, \{2,3\}, \{2,3,4\}\}$ & 15 & 5 \\
25 & $\{\emptyset, \{1\}, \{1,2\}, \{2,3\}, \{2,3,4\}\}$ & 18 & 5 & 117 & $\{\{1\}, \{1,2\}, \{1,3\}, \{2,3\}, \{1,2,3,4\}\}$ & 20 & 5 \\
26 & $\{\emptyset, \{1\}, \{1,2\}, \{2,3\}, \{1,2,3,4\}\}$ & 17 & 5 & 118 & $\{\{1\}, \{1,2\}, \{1,3\}, \{1,4\}, \{2,3,4\}\}$ & 24 & 6 \\
27 & $\{\emptyset, \{1\}, \{1,2\}, \{1,2,3\}, \{1,2,4\}\}$ & 12 & 5 & 119 & $\{\{1\}, \{1,2\}, \{1,3\}, \{2,4\}, \{1,2,4\}\}$ & 6 & 3 \\
28 & $\{\emptyset, \{1\}, \{1,2\}, \{1,2,3\}, \{3,4\}\}$ & 13 & 5 & 120 & $\{\{1\}, \{1,2\}, \{1,3\}, \{2,4\}, \{1,3,4\}\}$ & 14 & 5 \\
29 & $\{\emptyset, \{1\}, \{1,2\}, \{1,2,3\}, \{1,3,4\}\}$ & 17 & 6 & 121 & $\{\{1\}, \{1,2\}, \{1,3\}, \{2,4\}, \{1,2,3,4\}\}$ & 14 & 5 \\
30 & $\{\emptyset, \{1\}, \{1,2\}, \{1,2,3\}, \{2,3,4\}\}$ & 13 & 5 & 122 & $\{\{1\}, \{1,2\}, \{1,3\}, \{1,2,4\}, \{2,3,4\}\}$ & 14 & 5 \\
31 & $\{\emptyset, \{1\}, \{1,2\}, \{1,2,3\}, \{1,2,3,4\}\}$ & 4 & 2 & 123 & $\{\{1\}, \{1,2\}, \{1,3\}, \{2,3,4\}, \{1,2,3,4\}\}$ & 12 & 5 \\
32 & $\{\emptyset, \{1\}, \{2,3\}, \{2,4\}, \{3,4\}\}$ & 80 & 8 & 124 & $\{\{1\}, \{1,2\}, \{2,3\}, \{1,2,3\}, \{2,4\}\}$ & 18 & 6 \\
33 & $\{\emptyset, \{1\}, \{2,3\}, \{2,4\}, \{1,3,4\}\}$ & 51 & 7 & 125 & $\{\{1\}, \{1,2\}, \{2,3\}, \{1,2,3\}, \{1,2,4\}\}$ & 13 & 5 \\
34 & $\{\emptyset, \{1\}, \{2,3\}, \{2,4\}, \{2,3,4\}\}$ & 32 & 8 & 126 & $\{\{1\}, \{1,2\}, \{2,3\}, \{1,2,3\}, \{3,4\}\}$ & 14 & 5 \\
35 & $\{\emptyset, \{1\}, \{2,3\}, \{2,4\}, \{1,2,3,4\}\}$ & 40 & 8 & 127 & $\{\{1\}, \{1,2\}, \{2,3\}, \{1,2,3\}, \{1,3,4\}\}$ & 17 & 6 \\
36 & $\{\emptyset, \{1\}, \{2,3\}, \{1,2,4\}, \{1,3,4\}\}$ & 38 & 7 & 128 & $\{\{1\}, \{1,2\}, \{2,3\}, \{1,2,3\}, \{2,3,4\}\}$ & 7 & 3 \\
37 & $\{\emptyset, \{1\}, \{2,3\}, \{1,2,4\}, \{2,3,4\}\}$ & 23 & 5 & 129 & $\{\{1\}, \{1,2\}, \{2,3\}, \{1,2,3\}, \{1,2,3,4\}\}$ & 10 & 4 \\
38 & $\{\emptyset, \{1\}, \{2,3\}, \{1,2,4\}, \{1,2,3,4\}\}$ & 23 & 5 & 130 & $\{\{1\}, \{1,2\}, \{2,3\}, \{2,4\}, \{3,4\}\}$ & 33 & 7 \\
39 & $\{\emptyset, \{1\}, \{1,2,3\}, \{1,2,4\}, \{1,3,4\}\}$ & 38 & 7 & 131 & $\{\{1\}, \{1,2\}, \{2,3\}, \{2,4\}, \{1,3,4\}\}$ & 38 & 7 \\
40 & $\{\emptyset, \{1\}, \{1,2,3\}, \{1,2,4\}, \{2,3,4\}\}$ & 28 & 6 & 132 & $\{\{1\}, \{1,2\}, \{2,3\}, \{2,4\}, \{2,3,4\}\}$ & 20 & 6 \\
41 & $\{\emptyset, \{1\}, \{1,2,3\}, \{1,2,4\}, \{1,2,3,4\}\}$ & 20 & 6 & 133 & $\{\{1\}, \{1,2\}, \{2,3\}, \{2,4\}, \{1,2,3,4\}\}$ & 30 & 6 \\
42 & $\{\emptyset, \{1,2\}, \{1,3\}, \{2,3\}, \{1,4\}\}$ & 48 & 8 & 134 & $\{\{1\}, \{1,2\}, \{2,3\}, \{1,2,4\}, \{1,3,4\}\}$ & 27 & 6 \\
43 & $\{\emptyset, \{1,2\}, \{1,3\}, \{2,3\}, \{1,2,4\}\}$ & 42 & 7 & 135 & $\{\{1\}, \{1,2\}, \{2,3\}, \{1,2,4\}, \{1,2,3,4\}\}$ & 19 & 5 \\
44 & $\{\emptyset, \{1,2\}, \{1,3\}, \{2,3\}, \{1,2,3,4\}\}$ & 20 & 6 & 136 & $\{\{1\}, \{1,2\}, \{2,3\}, \{3,4\}, \{1,3,4\}\}$ & 23 & 5 \\
45 & $\{\emptyset, \{1,2\}, \{1,3\}, \{1,2,3\}, \{1,4\}\}$ & 48 & 9 & 137 & $\{\{1\}, \{1,2\}, \{2,3\}, \{3,4\}, \{1,2,3,4\}\}$ & 26 & 6 \\
46 & $\{\emptyset, \{1,2\}, \{1,3\}, \{1,2,3\}, \{2,4\}\}$ & 20 & 6 & 138 & $\{\{1\}, \{1,2\}, \{2,3\}, \{1,3,4\}, \{2,3,4\}\}$ & 23 & 5 \\
47 & $\{\emptyset, \{1,2\}, \{1,3\}, \{1,2,3\}, \{1,2,4\}\}$ & 18 & 6 & 139 & $\{\{1\}, \{1,2\}, \{2,3\}, \{2,3,4\}, \{1,2,3,4\}\}$ & 13 & 5 \\
48 & $\{\emptyset, \{1,2\}, \{1,3\}, \{1,2,3\}, \{2,3,4\}\}$ & 31 & 6 & 140 & $\{\{1\}, \{1,2\}, \{1,2,3\}, \{1,2,4\}, \{3,4\}\}$ & 24 & 6 \\
49 & $\{\emptyset, \{1,2\}, \{1,3\}, \{1,2,3\}, \{1,2,3,4\}\}$ & 10 & 5 & 141 & $\{\{1\}, \{1,2\}, \{1,2,3\}, \{1,2,4\}, \{2,3,4\}\}$ & 12 & 4 \\
50 & $\{\emptyset, \{1,2\}, \{1,3\}, \{1,4\}, \{2,3,4\}\}$ & 103 & 10 & 142 & $\{\{1\}, \{1,2\}, \{1,2,3\}, \{3,4\}, \{1,3,4\}\}$ & 18 & 5 \\
51 & $\{\emptyset, \{1,2\}, \{1,3\}, \{1,4\}, \{1,2,3,4\}\}$ & 40 & 8 & 143 & $\{\{1\}, \{1,2\}, \{1,2,3\}, \{3,4\}, \{1,2,3,4\}\}$ & 14 & 5 \\
52 & $\{\emptyset, \{1,2\}, \{1,3\}, \{2,4\}, \{1,3,4\}\}$ & 33 & 6 & 144 & $\{\{1\}, \{1,2\}, \{1,2,3\}, \{1,3,4\}, \{2,3,4\}\}$ & 17 & 5 \\
53 & $\{\emptyset, \{1,2\}, \{1,3\}, \{1,2,4\}, \{2,3,4\}\}$ & 43 & 7 & 145 & $\{\{1\}, \{1,2\}, \{1,2,3\}, \{2,3,4\}, \{1,2,3,4\}\}$ & 5 & 3 \\
54 & $\{\emptyset, \{1,2\}, \{1,3\}, \{1,2,4\}, \{1,2,3,4\}\}$ & 15 & 6 & 146 & $\{\{1\}, \{2,3\}, \{1,2,3\}, \{2,4\}, \{3,4\}\}$ & 41 & 7 \\
55 & $\{\emptyset, \{1,2\}, \{1,3\}, \{2,3,4\}, \{1,2,3,4\}\}$ & 30 & 6 & 147 & $\{\{1\}, \{2,3\}, \{1,2,3\}, \{2,4\}, \{1,3,4\}\}$ & 38 & 7 \\
56 & $\{\emptyset, \{1,2\}, \{1,2,3\}, \{1,2,4\}, \{1,3,4\}\}$ & 25 & 6 & 148 & $\{\{1\}, \{2,3\}, \{1,2,3\}, \{2,4\}, \{2,3,4\}\}$ & 15 & 5 \\
57 & $\{\emptyset, \{1,2\}, \{1,2,3\}, \{3,4\}, \{1,3,4\}\}$ & 23 & 5 & 149 & $\{\{1\}, \{2,3\}, \{1,2,3\}, \{2,4\}, \{1,2,3,4\}\}$ & 28 & 6 \\
58 & $\{\emptyset, \{1,2\}, \{1,2,3\}, \{1,3,4\}, \{2,3,4\}\}$ & 36 & 7 & 150 & $\{\{1\}, \{2,3\}, \{1,2,3\}, \{1,2,4\}, \{1,3,4\}\}$ & 36 & 7 \\
59 & $\{\emptyset, \{1,2\}, \{1,2,3\}, \{1,3,4\}, \{1,2,3,4\}\}$ & 14 & 5 & 151 & $\{\{1\}, \{2,3\}, \{1,2,3\}, \{1,2,4\}, \{2,3,4\}\}$ & 18 & 6 \\
60 & $\{\emptyset, \{1,2,3\}, \{1,2,4\}, \{1,3,4\}, \{2,3,4\}\}$ & 38 & 7 & 152 & $\{\{1\}, \{2,3\}, \{1,2,3\}, \{1,2,4\}, \{1,2,3,4\}\}$ & 21 & 6 \\
61 & $\{\emptyset, \{1,2,3\}, \{1,2,4\}, \{1,3,4\}, \{1,2,3,4\}\}$ & 20 & 6 & 153 & $\{\{1\}, \{2,3\}, \{2,4\}, \{3,4\}, \{2,3,4\}\}$ & 25 & 7 \\
62 & $\{\{1\}, \{2\}, \{1,2\}, \{3\}, \{4\}\}$ & 72 & 9 & 154 & $\{\{1\}, \{2,3\}, \{2,4\}, \{3,4\}, \{1,2,3,4\}\}$ & 56 & 8 \\
63 & $\{\{1\}, \{2\}, \{1,2\}, \{3\}, \{1,4\}\}$ & 31 & 7 & 155 & $\{\{1\}, \{2,3\}, \{2,4\}, \{1,3,4\}, \{1,2,3,4\}\}$ & 37 & 6 \\
64 & $\{\{1\}, \{2\}, \{1,2\}, \{3\}, \{1,2,4\}\}$ & 10 & 5 & 156 & $\{\{1\}, \{2,3\}, \{2,4\}, \{2,3,4\}, \{1,2,3,4\}\}$ & 18 & 6 \\
65 & $\{\{1\}, \{2\}, \{1,2\}, \{3\}, \{3,4\}\}$ & 35 & 7 & 157 & $\{\{1\}, \{2,3\}, \{1,2,4\}, \{1,3,4\}, \{2,3,4\}\}$ & 27 & 6 \\
66 & $\{\{1\}, \{2\}, \{1,2\}, \{3\}, \{1,3,4\}\}$ & 18 & 5 & 158 & $\{\{1\}, \{2,3\}, \{1,2,4\}, \{1,3,4\}, \{1,2,3,4\}\}$ & 29 & 6 \\
67 & $\{\{1\}, \{2\}, \{1,2\}, \{3\}, \{1,2,3,4\}\}$ & 27 & 6 & 159 & $\{\{1\}, \{2,3\}, \{1,2,4\}, \{2,3,4\}, \{1,2,3,4\}\}$ & 15 & 5 \\
68 & $\{\{1\}, \{2\}, \{1,2\}, \{1,3\}, \{1,4\}\}$ & 13 & 5 & 160 & $\{\{1\}, \{1,2,3\}, \{1,2,4\}, \{1,3,4\}, \{2,3,4\}\}$ & 24 & 6 \\
69 & $\{\{1\}, \{2\}, \{1,2\}, \{1,3\}, \{2,4\}\}$ & 8 & 3 & 161 & $\{\{1\}, \{1,2,3\}, \{1,2,4\}, \{2,3,4\}, \{1,2,3,4\}\}$ & 12 & 5 \\
70 & $\{\{1\}, \{2\}, \{1,2\}, \{1,3\}, \{1,2,4\}\}$ & 11 & 4 & 162 & $\{\{1,2\}, \{1,3\}, \{2,3\}, \{1,2,3\}, \{1,4\}\}$ & 12 & 5 \\
71 & $\{\{1\}, \{2\}, \{1,2\}, \{1,3\}, \{3,4\}\}$ & 22 & 8 & 163 & $\{\{1,2\}, \{1,3\}, \{2,3\}, \{1,2,3\}, \{1,2,4\}\}$ & 11 & 4 \\
72 & $\{\{1\}, \{2\}, \{1,2\}, \{1,3\}, \{1,3,4\}\}$ & 7 & 3 & 164 & $\{\{1,2\}, \{1,3\}, \{2,3\}, \{1,2,3\}, \{1,2,3,4\}\}$ & 4 & 2 \\
73 & $\{\{1\}, \{2\}, \{1,2\}, \{1,3\}, \{2,3,4\}\}$ & 15 & 5 & 165 & $\{\{1,2\}, \{1,3\}, \{2,3\}, \{1,4\}, \{1,2,4\}\}$ & 18 & 5 \\
74 & $\{\{1\}, \{2\}, \{1,2\}, \{1,3\}, \{1,2,3,4\}\}$ & 16 & 5 & 166 & $\{\{1,2\}, \{1,3\}, \{2,3\}, \{1,4\}, \{2,3,4\}\}$ & 27 & 6 \\
75 & $\{\{1\}, \{2\}, \{1,2\}, \{1,2,3\}, \{1,2,4\}\}$ & 15 & 5 & 167 & $\{\{1,2\}, \{1,3\}, \{2,3\}, \{1,4\}, \{1,2,3,4\}\}$ & 12 & 5 \\
76 & $\{\{1\}, \{2\}, \{1,2\}, \{1,2,3\}, \{3,4\}\}$ & 30 & 7 & 168 & $\{\{1,2\}, \{1,3\}, \{2,3\}, \{1,2,4\}, \{1,2,3,4\}\}$ & 13 & 5 \\
77 & $\{\{1\}, \{2\}, \{1,2\}, \{1,2,3\}, \{1,3,4\}\}$ & 9 & 5 & 169 & $\{\{1,2\}, \{1,3\}, \{1,2,3\}, \{1,4\}, \{2,3,4\}\}$ & 27 & 6 \\
78 & $\{\{1\}, \{2\}, \{1,2\}, \{1,2,3\}, \{1,2,3,4\}\}$ & 8 & 4 & 170 & $\{\{1,2\}, \{1,3\}, \{1,2,3\}, \{2,4\}, \{1,2,4\}\}$ & 8 & 3 \\
79 & $\{\{1\}, \{2\}, \{3\}, \{1,2,3\}, \{4\}\}$ & 40 & 8 & 171 & $\{\{1,2\}, \{1,3\}, \{1,2,3\}, \{2,4\}, \{1,3,4\}\}$ & 15 & 5 \\
80 & $\{\{1\}, \{2\}, \{3\}, \{1,2,3\}, \{1,4\}\}$ & 54 & 7 & 172 & $\{\{1,2\}, \{1,3\}, \{1,2,3\}, \{2,4\}, \{1,2,3,4\}\}$ & 5 & 3 \\
81 & $\{\{1\}, \{2\}, \{3\}, \{1,2,3\}, \{1,2,4\}\}$ & 23 & 6 & 173 & $\{\{1,2\}, \{1,3\}, \{1,2,3\}, \{1,2,4\}, \{2,3,4\}\}$ & 16 & 5 \\
82 & $\{\{1\}, \{2\}, \{3\}, \{1,2,3\}, \{1,2,3,4\}\}$ & 24 & 6 & 174 & $\{\{1,2\}, \{1,3\}, \{1,2,3\}, \{2,3,4\}, \{1,2,3,4\}\}$ & 9 & 4 \\
83 & $\{\{1\}, \{2\}, \{3\}, \{4\}, \{1,2,3,4\}\}$ & 89 & 10 & 175 & $\{\{1,2\}, \{1,3\}, \{1,4\}, \{2,3,4\}, \{1,2,3,4\}\}$ & 24 & 6 \\
84 & $\{\{1\}, \{2\}, \{3\}, \{1,4\}, \{1,2,4\}\}$ & 30 & 7 & 176 & $\{\{1,2\}, \{1,3\}, \{2,4\}, \{1,3,4\}, \{2,3,4\}\}$ & 23 & 5 \\
85 & $\{\{1\}, \{2\}, \{3\}, \{1,4\}, \{2,3,4\}\}$ & 55 & 7 & 177 & $\{\{1,2\}, \{1,3\}, \{2,4\}, \{1,3,4\}, \{1,2,3,4\}\}$ & 11 & 5 \\
86 & $\{\{1\}, \{2\}, \{3\}, \{1,4\}, \{1,2,3,4\}\}$ & 59 & 7 & 178 & $\{\{1,2\}, \{1,3\}, \{1,2,4\}, \{2,3,4\}, \{1,2,3,4\}\}$ & 13 & 5 \\
87 & $\{\{1\}, \{2\}, \{3\}, \{1,2,4\}, \{1,2,3,4\}\}$ & 26 & 6 & 179 & $\{\{1,2\}, \{1,2,3\}, \{1,2,4\}, \{3,4\}, \{1,3,4\}\}$ & 15 & 5 \\
88 & $\{\{1\}, \{2\}, \{1,3\}, \{1,2,3\}, \{1,4\}\}$ & 15 & 5 & 180 & $\{\{1,2\}, \{1,2,3\}, \{1,2,4\}, \{1,3,4\}, \{2,3,4\}\}$ & 20 & 5 \\
89 & $\{\{1\}, \{2\}, \{1,3\}, \{1,2,3\}, \{2,4\}\}$ & 13 & 5 & 181 & $\{\{1,2\}, \{1,2,3\}, \{3,4\}, \{1,3,4\}, \{1,2,3,4\}\}$ & 6 & 3 \\
90 & $\{\{1\}, \{2\}, \{1,3\}, \{1,2,3\}, \{1,2,4\}\}$ & 28 & 7 & 182 & $\{\{1,2\}, \{1,2,3\}, \{1,3,4\}, \{2,3,4\}, \{1,2,3,4\}\}$ & 10 & 4 \\
91 & $\{\{1\}, \{2\}, \{1,3\}, \{1,2,3\}, \{3,4\}\}$ & 37 & 7 & 183 & $\{\{1,2,3\}, \{1,2,4\}, \{1,3,4\}, \{2,3,4\}, \{1,2,3,4\}\}$ & 8 & 3 \\
92 & $\{\{1\}, \{2\}, \{1,3\}, \{1,2,3\}, \{1,3,4\}\}$ & 9 & 4 & & & & \\
\hline
\end{tabular}
\end{table}
\FloatBarrier

\section{Sample Bernstein Coefficients for Orbit $164$ in Section \ref{subsec:illustrative_subdivision_example}}
\label{appendix_sec:sample_bernstein_tables}
\thispagestyle{empty}
\footnotesize
\setlength{\tabcolsep}{3pt}
\renewcommand{\arraystretch}{0.88}


\begin{minipage}{\textwidth}
\captionsetup[table]{skip=2pt}
\captionof{table}{Bernstein coefficients for sub-simplex 1.}
\label{tab:our_leaf_1}

\begin{center}
\noindent\begin{tabularx}{\linewidth}{
>{\raggedright\arraybackslash}p{0.08\linewidth}
>{\raggedright\arraybackslash}p{0.25\linewidth}
>{\raggedright\arraybackslash}X
>{\centering\arraybackslash}p{0.07\linewidth}}
\toprule
Block & $\mb{p}^{ab}$ & $\mb{\gamma}^{ab}$ & $c^{ab}$ \\
\midrule
$(0,0)$ & $\{\,1:1/4\; 3:1/4\; 5:1/4\; 7:1/4\,\}$ & $\{\,0:1/3\; 10:1/6\; 15:1/6\,\}$ & $1/3$ \\
$(0,1)$ & $\{\,0:1/8\; 3:1/8\; 5:3/8\; 6:1/8\; 7:1/4\,\}$ & $\{\,0:1/6\; 4:1/6\; 10:1/4\; 17:1/12\,\}$ & $1/3$ \\
$(0,2)$ & $\{\,2:1/4\; 4:1/4\; 7:1/2\,\}$ & $\{\,10:1/4\; 15:1/4\; 17:1/6\; 35:1/3\,\}$ & $1/3$ \\
$(0,3)$ & $\{\,3:1/4\; 5:1/4\; 7:1/2\,\}$ & $\{\,0:1/3\; 4:1/4\; 5:1/12\; 10:1/6\,\}$ & $1/3$ \\
$(0,4)$ & $\{\,7:1/2\; 11:1/4\; 13:1/4\,\}$ & $\{\,1:1/12\; 2:1/4\; 14:1/12\; 15:1/6\; 16:1/12\; 17:1/4\; 25:1/12\; 43:1/12\; 44:1/4\; 47:1/6\; 55:1/4\,\}$ & $1/3$ \\
$(1,1)$ & $\{\,4:1/4\; 5:1/4\; 6:1/4\; 7:1/4\,\}$ & $\{\,2:1/3\; 15:1/6\; 17:1/6\,\}$ & $1/3$ \\
$(1,2)$ & $\{\,4:1/4\; 6:1/2\; 7:1/4\,\}$ & $\{\,2:1/3\; 15:1/4\; 17:1/12\,\}$ & $1/3$ \\
$(1,3)$ & $\{\,5:1/4\; 6:1/4\; 7:1/2\,\}$ & $\{\,0:1/4\; 1:1/12\; 4:1/6\; 5:1/12\; 8:1/12\; 10:1/6\,\}$ & $1/3$ \\
$(1,4)$ & $\{\,7:1/2\; 13:1/4\; 14:1/4\,\}$ & $\{\,2:1/6\; 14:1/6\; 15:1/4\; 17:1/4\; 22:1/6\; 34:1/6\; 40:1/6\; 48:1/6\; 55:1/3\,\}$ & $1/3$ \\
$(2,2)$ & $\{\,6:1\,\}$ & $\{\,1:1/3\; 8:1/3\,\}$ & $1/3$ \\
$(2,3)$ & $\{\,6:1/2\; 7:1/2\,\}$ & $\{\,0:1/6\; 1:1/6\; 4:1/6\; 8:1/6\; 10:1/6\,\}$ & $1/3$ \\
$(2,4)$ & $\{\,7:1/2\; 14:1/2\,\}$ & $\{\,17:1/6\; 18:1/6\; 21:1/3\; 26:1/3\; 34:1/3\; 39:1/3\; 44:1/3\; 49:1/6\,\}$ & $1/3$ \\
$(3,3)$ & $\{\,7:1\,\}$ & $\{\,0:1/3\; 4:1/3\; 10:1/3\,\}$ & $1/3$ \\
$(3,4)$ & $\{\,7:1/2\; 15:1/2\,\}$ & $\{\,0:1/3\; 4:1/3\; 10:1/3\; 19:1/6\,\}$ & $1/3$ \\
$(4,4)$ & $\{\,15:1\,\}$ & $\{\,0:1/3\; 4:1/3\; 10:1/3\; 19:1/3\,\}$ & $1/3$ \\
\bottomrule
\end{tabularx}
\end{center}
\end{minipage}
\vspace{2em}


\begin{minipage}{\textwidth}
\captionsetup{skip=2pt}
\captionof{table}{Bernstein coefficients for sub-simplex 2.}
\label{tab:our_leaf_2}
\begin{center}
\hspace{-1.5cm}
\noindent\begin{tabularx}{\linewidth}{
>{\raggedright\arraybackslash}p{0.08\linewidth}
>{\raggedright\arraybackslash}p{0.25\linewidth}
>{\raggedright\arraybackslash}X
>{\centering\arraybackslash}p{0.07\linewidth}}
\toprule
Block & $\mb{p}^{ab}$ & $\mb{\gamma}^{ab}$ & $c^{ab}$ \\
\midrule
$(0,0)$ & $\{\,1:1/4\; 3:1/4\; 5:1/4\; 7:1/4\,\}$ & $\{\,0:1/3\; 10:1/6\; 15:1/6\,\}$ & $1/3$ \\
$(0,1)$ & $\{\,1:1/4\; 5:1/2\; 7:1/4\,\}$ & $\{\,0:1/3\; 10:1/4\; 15:1/12\,\}$ & $1/3$ \\
$(0,2)$ & $\{\,0:1/8\; 3:1/8\; 5:3/8\; 6:1/8\; 7:1/4\,\}$ & $\{\,0:1/6\; 4:1/6\; 10:1/4\; 17:1/12\,\}$ & $1/3$ \\
$(0,3)$ & $\{\,3:1/4\; 5:1/4\; 7:1/2\,\}$ & $\{\,0:1/3\; 4:1/4\; 5:1/12\; 10:1/6\,\}$ & $1/3$ \\
$(0,4)$ & $\{\,7:1/2\; 11:1/4\; 13:1/4\,\}$ & $\{\,1:1/12\; 2:1/4\; 14:1/12\; 15:1/6\; 16:1/12\; 17:1/4\; 25:1/12\; 43:1/12\; 44:1/4\; 47:1/6\; 55:1/4\,\}$ & $1/3$ \\
$(1,1)$ & $\{\,5:1\,\}$ & $\{\,0:1/3\; 5:1/3\,\}$ & $1/3$ \\
$(1,2)$ & $\{\,4:1/4\; 5:1/2\; 7:1/4\,\}$ & $\{\,2:1/3\; 15:1/12\; 17:1/4\,\}$ & $1/3$ \\
$(1,3)$ & $\{\,5:1/2\; 7:1/2\,\}$ & $\{\,0:1/3\; 4:1/6\; 5:1/6\; 10:1/6\,\}$ & $1/3$ \\
$(1,4)$ & $\{\,7:1/2\; 13:1/2\,\}$ & $\{\,0:1/3\; 15:1/6\; 16:1/6\; 24:1/3\; 39:1/3\; 44:1/3\; 48:1/6\,\}$ & $1/3$ \\
$(2,2)$ & $\{\,4:1/4\; 5:1/4\; 6:1/4\; 7:1/4\,\}$ & $\{\,2:1/3\; 15:1/6\; 17:1/6\,\}$ & $1/3$ \\
$(2,3)$ & $\{\,5:1/4\; 6:1/4\; 7:1/2\,\}$ & $\{\,0:1/4\; 1:1/12\; 4:1/6\; 5:1/12\; 8:1/12\; 10:1/6\,\}$ & $1/3$ \\
$(2,4)$ & $\{\,7:1/2\; 13:1/4\; 14:1/4\,\}$ & $\{\,2:1/6\; 14:1/6\; 15:1/4\; 17:1/4\; 22:1/6\; 34:1/6\; 40:1/6\; 48:1/6\; 55:1/3\,\}$ & $1/3$ \\
$(3,3)$ & $\{\,7:1\,\}$ & $\{\,0:1/3\; 4:1/3\; 10:1/3\,\}$ & $1/3$ \\
$(3,4)$ & $\{\,7:1/2\; 15:1/2\,\}$ & $\{\,0:1/3\; 4:1/3\; 10:1/3\; 19:1/6\,\}$ & $1/3$ \\
$(4,4)$ & $\{\,15:1\,\}$ & $\{\,0:1/3\; 4:1/3\; 10:1/3\; 19:1/3\,\}$ & $1/3$ \\
\bottomrule
\end{tabularx}
\end{center}
\end{minipage}
\vspace{2em}


\begin{minipage}{\textwidth}
\captionsetup[table]{skip=2pt}
\captionof{table}{Bernstein coefficients for sub-simplex 3.}
\label{tab:our_leaf_3}

\begin{center}
\hspace{-1.5cm}
\noindent\begin{tabularx}{\linewidth}{
>{\raggedright\arraybackslash}p{0.08\linewidth}
>{\raggedright\arraybackslash}p{0.25\linewidth}
>{\raggedright\arraybackslash}X
>{\centering\arraybackslash}p{0.07\linewidth}}
\toprule
Block & $\mb{p}^{ab}$ & $\mb{\gamma}^{ab}$ & $c^{ab}$ \\
\midrule
$(0,0)$ & $\{\,2:1/4\; 3:1/4\; 6:1/4\; 7:1/4\,\}$ & $\{\,1:1/3\; 10:1/6\; 17:1/6\,\}$ & $1/3$ \\
$(0,1)$ & $\{\,0:1/8\; 3:3/8\; 5:1/8\; 6:1/8\; 7:1/4\,\}$ & $\{\,0:1/6\; 4:1/6\; 10:1/6\; 15:1/12\; 17:1/12\,\}$ & $1/3$ \\
$(0,2)$ & $\{\,2:1/4\; 6:1/2\; 7:1/4\,\}$ & $\{\,1:1/3\; 10:1/4\; 17:1/12\,\}$ & $1/3$ \\
$(0,3)$ & $\{\,3:1/4\; 6:1/4\; 7:1/2\,\}$ & $\{\,0:1/4\; 1:1/12\; 4:1/4\; 8:1/12\; 10:1/6\,\}$ & $1/3$ \\
$(0,4)$ & $\{\,7:1/2\; 11:1/4\; 14:1/4\,\}$ & $\{\,0:1/12\; 2:1/12\; 14:1/12\; 15:1/4\; 17:1/6\; 18:1/12\; 22:1/6\; 23:1/12\; 34:1/6\; 40:1/4\; 44:1/4\; 55:1/4\,\}$ & $1/3$ \\
$(1,1)$ & $\{\,1:1/4\; 3:1/4\; 5:1/4\; 7:1/4\,\}$ & $\{\,0:1/3\; 10:1/6\; 15:1/6\,\}$ & $1/3$ \\
$(1,2)$ & $\{\,2:1/4\; 4:1/4\; 7:1/2\,\}$ & $\{\,10:1/4\; 15:1/4\; 17:1/6\; 35:1/3\,\}$ & $1/3$ \\
$(1,3)$ & $\{\,3:1/4\; 5:1/4\; 7:1/2\,\}$ & $\{\,0:1/3\; 4:1/4\; 5:1/12\; 10:1/6\,\}$ & $1/3$ \\
$(1,4)$ & $\{\,7:1/2\; 11:1/4\; 13:1/4\,\}$ & $\{\,1:1/12\; 2:1/4\; 14:1/12\; 15:1/6\; 16:1/12\; 17:1/4\; 25:1/12\; 43:1/12\; 44:1/4\; 47:1/6\; 55:1/4\,\}$ & $1/3$ \\
$(2,2)$ & $\{\,6:1\,\}$ & $\{\,1:1/3\; 8:1/3\,\}$ & $1/3$ \\
$(2,3)$ & $\{\,6:1/2\; 7:1/2\,\}$ & $\{\,0:1/6\; 1:1/6\; 4:1/6\; 8:1/6\; 10:1/6\,\}$ & $1/3$ \\
$(2,4)$ & $\{\,7:1/2\; 14:1/2\,\}$ & $\{\,17:1/6\; 18:1/6\; 21:1/3\; 26:1/3\; 34:1/3\; 39:1/3\; 44:1/3\; 49:1/6\,\}$ & $1/3$ \\
$(3,3)$ & $\{\,7:1\,\}$ & $\{\,0:1/3\; 4:1/3\; 10:1/3\,\}$ & $1/3$ \\
$(3,4)$ & $\{\,7:1/2\; 15:1/2\,\}$ & $\{\,0:1/3\; 4:1/3\; 10:1/3\; 19:1/6\,\}$ & $1/3$ \\
$(4,4)$ & $\{\,15:1\,\}$ & $\{\,0:1/3\; 4:1/3\; 10:1/3\; 19:1/3\,\}$ & $1/3$ \\
\bottomrule
\end{tabularx}
\end{center}
\end{minipage}
\restoregeometry

\end{appendices}
\end{document}